\documentclass[11pt,oneside,a4paper,reqno]{amsart}
\RequirePackage{fix-cm} 
\usepackage[T1]{fontenc}
\usepackage[english]{babel}
\usepackage{amssymb,amsmath,amsthm,graphicx,amsfonts}
\usepackage{lipsum}

\usepackage[many]{tcolorbox}

\usepackage[
  rm={oldstyle=false,proportional=true},
  sf={oldstyle=true,proportional=true},
  tt={oldstyle=true,proportional=true,variable=true},
  qt=false,
]{cfr-lm}
\usepackage{cite,enumerate,setspace}
\usepackage{hyperref}
\hypersetup{
colorlinks,
citecolor=black,
linkcolor=black,
urlcolor=black}

\usepackage{geometry}
\makeatletter
\def\@seccntformat#1{\hspace*{0mm}%
 \protect\textup{\protect\@secnumfont
   \ifnum\pdfstrcmp{subsection}{#1}=0 \bfseries\fi
   \csname the#1\endcsname
   \protect\@secnumpunct
     }%
}
\usepackage[mathlines]{lineno}

\usepackage{etoolbox} 

\newtheorem{definition}{\sbweight Definition}
\newtheorem{lemma}{Lemma}

\newtheorem{theorem}{\sbweight Theorem}
\theoremstyle{remark}
\newtheorem{remark}{\sc  Remark\rm}[section]

\catcode`\>=\active \def>{
\fontencoding{T1}\selectfont\symbol{62}\fontencoding{\encodingdefault}}
\newcommand{\assign}{:=}

\newcommand{\nobracket}{}
\newcommand{\of}{:}
\newcommand{\tmabbr}[1]{#1}

\newcommand{\tmem}[1]{{\em #1\/}}
\newcommand{\tmmathbf}[1]{\ensuremath{\boldsymbol{#1}}}

\newcommand{\tmname}[1]{\textsc{#1}}
\newcommand{\tmop}[1]{{\operatorname{#1}}}

\newcommand{\tmstrong}[1]{\textbf{#1}}

\newenvironment{enumerateroman}{\begin{enumerate}[i.] }{\end{enumerate}}

\newcommand{\acro}[1]{#1}
\newcommand{\jname}[1]{{\tmname{#1}}}
\newcommand{\sdllg}{{SDLLG}}
\newcommand{\sllg}{{SLLG}}
\newcommand{\dt}[1]{\partial_t #1}
\newcommand{\heff}{\tmmathbf{h}_{\tmop{eff}}}
\newcommand{\hd}{\tmmathbf{h}_{\tmop{d}}}
\newcommand{\phian}{\phi_{\tmop{an}}}
\newcommand{\cex}{c_{\tmop{ex}}}
\newcommand{\Omp}{\Omega'}
\newcommand{\gr}{\gamma_0}
\newcommand{\Dzero}{D_0}
\newcommand{\pmone}{\tmmathbf{\pi}_{\m_1} [\tmmathbf{v}]}
\newcommand{\Kc}{\mathcal{K}}
\newcommand{\divv}{\mathrm{div}\,}

\newcommand{\divb}{\mathbf{div}\,}
\newcommand{\Fsp}{\mathbf{{\of}}}
\newcommand{\lapl}{\Delta}
\newcommand{\Ti}{T}

\newcommand{\Om}{\Omega}
\newcommand{\T}{\top}

\newcommand{\m}{\boldsymbol{m}}
\newcommand{\s}{\boldsymbol{s}}
\newcommand{\eqs}{=}

\newcommand{\RR}{\mathbb{R}}
\newcommand{\Stwo}{\mathbb{S}}

\newcommand{\grad}{\nabla}

\DeclareMathOperator*{\essinf}{ess\,inf}

\newcounter{myhypo}
\renewcommand\themyhypo{(H\arabic{myhypo})}

\newtcolorbox{hypo}[1][]{
  breakable,
  enhanced,
  top=0pt,
  bottom=0pt,
  nobeforeafter,
  colback=white,
  boxrule=0pt,
  arc=0pt,
  right=30pt,
  left=20pt,
  outer arc=0pt,
  overlay={
    \node[inner sep=0pt,anchor=east] 
    at (frame.east) 
    {\refstepcounter{myhypo}\themyhypo\label{#1}};
  },
}
\newenvironment{hypotheses}
  {\list{}{\setlength\leftmargin{0pt}\item\relax}}
  {\endlist}
\newcommand\Hypo[2][]{%
  \begin{hypo}[#1]#2\end{hypo}}

\begin{document}

\title[Spin-diffusion model for micromagnetics in the limit of
long times]{Spin-diffusion model for micromagnetics \protect\\ in the limit of
long times}

\author{Giovanni Di Fratta}
\address{Institute for Analysis and Scientific Computing, TU Wien, Wiedner
Hauptstrae 8-10, 1040 Wien, Austria}
\email{giovanni.difratta@asc.tuwien.ac.at}

\author{Ansgar J{\"u}ngel}
\address{Institute for Analysis and Scientific Computing, TU Wien, Wiedner
Hauptstrae 8-10, 1040 Wien, Austria}
\email{ansgar.juengel@tuwien.ac.at}

\author{Dirk Praetorius}
\address{Institute for Analysis and Scientific Computing, TU Wien, Wiedner
Hauptstrae 8-10, 1040 Wien, Austria}
\email{dirk.praetorius@asc.tuwien.ac.at}

\author{Valeriy Slastikov}
\address{School of Mathematics, University of Bristol, Bristol BS8 1TW, United
Kingdom}
\email{valeriy.slastikov@bristol.ac.uk}

\date{\today}

\begin{abstract}\sloppy
  In this paper, we consider spin-diffusion Landau--Lifshitz--Gilbert equations 
	(SDLLG), which consist of the time-dependent
  Landau--Lifshitz--Gilbert ({\acro{LLG}}) equation coupled with a time-dependent diffusion equation for the electron spin accumulation. 
  The model takes into account the diffusion process of
  the spin accumulation in the magnetization dynamics of ferromagnetic multilayers. We prove that in the
  limit of long times, the system reduces to simpler equations in which the
  {\acro{LLG}} equation is coupled to a nonlinear and nonlocal steady-state
  equation, referred to as {\sllg}. As a by-product, the existence of global weak solutions to the {\sllg} equation is obtained.
  Moreover, we prove  weak-strong uniqueness of solutions of {\sllg}, i.e., all weak solutions coincide with the
  (unique) strong solution as long as the latter exists in time. The results
  provide a solid mathematical ground to the qualitative behavior originally
  predicted by {\jname{Zhang}}, {\jname{Levy}}, and {\jname{Fert}} in
  {\cite{zhang2002mechanisms}} in ferromagnetic multilayers.
  

\smallskip
\noindent $2010$ \emph{Mathematics Subject Classification.} 35C20, 35D30, 35G20, 35G25 49S05.
\end{abstract}

{\maketitle}

\section{Introduction and Physical Motivations}

In the last decades, the study of ferromagnetic materials and their magnetization processes has
been the focus of considerable research for its
application to magnetic recording technology. Below the Curie temperature, ferromagnetic media possess a spontaneous magnetization state (which
is the result of a spontaneous alignment of the elementary magnetic moments in
the medium) that can be controlled through appropriate external magnetic
fields. The magnetic recording technology exploits the magnetization of
ferromagnetic media to store information.

The {\tmem{giant magnetoresistance}} effect ({\acro{GMR}}), for which {\jname{Fert}}
and {\jname{Gr{\"u}nberg}} have been awarded the Nobel prize in
2007, has introduced new solutions in the
design of magnetic random access memories ({\acro{MRAM}}s). In a typical
{\acro{MRAM}} device, the binary information is stored in elementary cells
that can be addressed via two perpendicular arrays of parallel conducting
lines, called word lines and bit lines. A schematic of
the {\acro{MRAM}} architecture is depicted in Figure~\ref{fig:mram}. 

The {\acro{GMR}} allows for a {\tmem{giant}} change in the resistance of a
conductor in response to an applied magnetic field, and it is the primary
mechanism behind the reading process in {\acro{MRAM}}s. Furthermore, the switching
(writing) process of an {\acro{MRAM}} cell can be achieved by magnetic field
pulses produced by the sum of horizontal and vertical currents. The magnetic
pulse induces a magnetic torque, whose strength depends on the angle between
the field and the magnetization, which permits the switching of the cell.

This behavior is conceptually simple, but it is tough to realize in practice on the nanoscale.
 One of the fundamental issues
connected with the downscaling of magnetic storage devices is the thermal
stability of magnetization states. In principle, the problem can be
circumvented by increasing the magnetic anisotropy of the material, but as a
consequence, higher magnetic fields are required to reverse the magnetization
states. However, these magnetic fields typically act on a
long-range, whereas it is desired that the field produced by the two currents
can switch only the target cell. For this reason, considerable attention has recently
been paid to design new strategies of magnetization switching in
which the applied field is assisted by additional external actions. Examples
of these new approaches are heat and microwave-assisted switchings
{\cite{KochA2000,daquino2011current}}.

\begin{figure}[tb]
\includegraphics[width=13.8178374655647cm,height=5.001443001443cm]{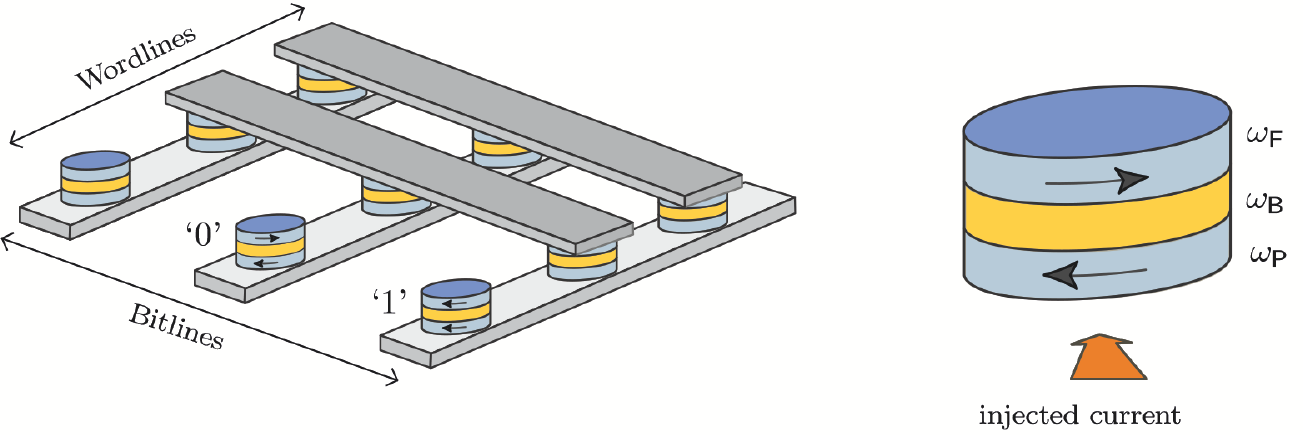}
  \caption{ \footnotesize \label{fig:mram}
  (Left) In a typical {\acro{MRAM}} device, the
  binary information is stored in elementary cells that can be addressed via
  two perpendicular arrays of parallel conducting lines, called word lines and
  bit lines. (Right) A schematic picture of an {\acro{MRAM}} cell. In its
simplest form, an {\acro{MRAM}} cell consists of two ferromagnetic layers $\omega_{\mathsf{F}}$ (the
{\tmem{free layer}}) and
$\omega_{\mathsf{P}}$ (the {\tmem{pinned layer}}), separated by a thin insulator
$\omega_{\mathsf{B}}$ (the \emph{tunnel barrier}). Such a multilayer structure is compatible with the {\acro{GMR}} effect. The logical states `$0$' and
  `$1$' are coded into the direction of the magnetization in the free layer
  $\omega_{\mathsf{F}}$. Note that \acro{MRAM} cells usually have an ellipsoidal cross-section. This is quite common in applications because ellipsoidal particles support single-domain magnetization states \cite{alouges2015liouville, BrownA1968, di2012generalization, DiFratta20160197} which allow for a qualitative analysis of the system dynamics in the framework of ODEs.
  }
\end{figure}

Fairly recently,
{\jname{Slonczewski}}~{\cite{slonczewski1996current}}
and {\jname{Berger}}~{\cite{berger1996emission}} proposed a
novel model for magnetization reversal based on the use of spin-polarized
currents injected directly in the magnetic free layer. The new mechanism
proved extremely valuable to overcome the difficulties imposed by the use of
strong magnetic fields. Since its introduction, it has been the object of
much research work in the spintronics community 
as a candidate to assist the switching of the magnetization in
{\acro{MRAM}}s cells. In this new approach, each {\acro{MRAM}} cell hosts a
multilayer structure, called {\tmem{magnetic tunnel junction}}, which in its
simplest form consists of two ferromagnetic layers $\omega_{\mathsf{F}}$ (the
so-called {\tmem{free layer}}, where the magnetization can change freely) and
$\omega_{\mathsf{P}}$ (the {\tmem{pinned layer}}, where the magnetization is
pinned by exchange interactions), separated by a thin insulator
$\omega_{\mathsf{B}}$ (the \emph{tunnel barrier}). A current is injected
perpendicular to the multilayer. The electron spin is polarized in the pinned
layer $\omega_{\mathsf{P}}$. When the electrons reach the free layer
$\omega_{\mathsf{F}}$, the spin exerts an additional torque on the underlying
magnetization, which assists the switching process.

The model proposed by
{\jname{Slonczewski}}~{\cite{slonczewski1996current}}
does not take into account the effects of spin diffusion, which have been
found to be important in understanding magnetoresistance in magnetic
multilayers. This motivated the work of
{\jname{Zhang}}, {\jname{Levy}}, and {\jname{Fert}}
{\cite{zhang2002mechanisms}}~(see also {\cite{shpiro2003self}}), where a new
spin-transfer model for the relaxation of the coupled system
spin-magnetization is proposed. Their model includes spatial variations in
both spin and magnetization, but is derived under the assumption that the
magnetization is uniform in each of the layers and, therefore, essentially
one-dimensional. Later, {\jname{Garc{\i}a-Cervera}} and {\jname{Wang}}
({\tmabbr{cf.}}~{\cite{garcia2007spin}}, see also~\cite{Abert2014})
generalized the model to the three-dimensional setting.

\subsection{The spin-diffusion Landau--Lifshitz--Gilbert equation}
In this section, we introduce the model proposed in
{\cite{garcia2007spin}}. For that, we set $\Om \assign
\omega_{\mathsf{F}} \cup \omega_{\mathsf{P}} \subset \RR^3$, and denote by $\Om' \assign \Om
\cup \omega_{\mathsf{B}}\subset \RR^3$, the region occupied by the multilayer. The spin
accumulation $\mathbf{S}$ is defined on $\Om'$, and
the magnetization $\mathbf{M}$ is supported on the region $\Om$ occupied by
the two magnetic layers. The magnetization is zero in $\omega_{\mathsf{B}}$
(i.e., outside the magnetic samples). We assume that the temperature is constant and well below the Curie temperature so that the magnetization $\mathbf{M}$ is of constant modulus in $\Om$, i.e., $|\mathbf{M}|=M_s$ with $M_s\in\RR_+$ being the saturation magnetization in $\Om$. To shorten notation, we set $\s=\mathbf{S}/{M_s}$ and $\m:=\mathbf{M}/M_s$. 

The spin-diffusion Landau--Lifshitz--Gilbert equation ({\sdllg}) consists of a
quasilinear diffusion equation for the evolution of the spin accumulation
coupled to the well-established equation for magnetization dynamics. In strong
form, {\sdllg} reads as
({\tmabbr{cf.}}~{\cite{Abert2014,garcia2007spin2,GarciaCervera2006}})
\begin{align}
  \dt{\s} & \eqs \divb \left[ J \left( \grad \s,
  \m \chi_{\Om} \right) \right] - \frac{2
  \Dzero}{\lambda^2_{\tmop{sf}}} \tmmathbf{s}- \frac{2 \Dzero}{\lambda^2_J}
  \s \times \m \chi_{\Om} \quad \text{in } \Omp \times \RR_+ , 
  \label{eq:spindifproc}\\
  \dt{\m} & \eqs - \gr \m \times \left( \heff
  [\m] + j_0 \s+\tmmathbf{f} \right) + \alpha \m
  \times \dt{\m} \quad \text{in } \Om \times \RR_+ .  \label{eq:LLGphysics}
\end{align}
Here, $\s: \Om' \times \RR_+ \rightarrow \RR^3$ is the spin
accumulation, $\m: \Om \times \RR_+ \rightarrow \Stwo^2$ is the
magnetization, 
and $\m \chi_{\Om}$ denotes the extension of
$\m$ by zero to the whole $\RR^3$. From now on, to simplify the notation, we
will identify $\m$ with its extension $\m \chi_{\Om}$ as
long as no ambiguities can arise. The first equation {\eqref{eq:spindifproc}}
models the spin-polarized transport in the multilayer $\Omp$ as a diffusive
process~{\cite{garcia2007spin,shpiro2003self,zhang2002mechanisms}}. The second
equation {\eqref{eq:LLGphysics}} is a variant of the Landau--Lifshitz--Gilbert
({\acro{LLG}}) equation {\cite{LandauA1935,gilbert2004phenomenological}} and
describes the relaxation process of the magnetization. Since the
modulus of the magnetization is  preserved by {\acro{LLG}}, we have
normalized the magnetization to take values on the 2-sphere $\Stwo^2$. The
system {\eqref{eq:spindifproc}}--{\eqref{eq:LLGphysics}} is supplemented with
the initial/boundary conditions
\begin{equation*}
  \left\{ \begin{array}{lll}
    \m (x, 0) \eqs \m^{\ast} (x) & \text{in} & \Om,\\
    \partial_{\tmmathbf{n}} \m \eqs 0 & \text{on} & \partial \Om \times  \RR_+,
  \end{array} \right. \qquad \left\{ \begin{array}{lll}
    \s (x, 0) \eqs \s^{\ast} (x) & \text{in} & \Omp,\\
    \partial_{\tmmathbf{n}} \s \eqs 0 & \text{on} & \partial \Omp \times  \RR_+,
  \end{array} \right. 
\end{equation*}
for some given initial states $\m^{\ast} : \Om \rightarrow \Stwo^2$,
$\s^{\ast} : \Om' \rightarrow \RR^3$. A detailed description of the
terms involved in {\sdllg} follows.

\subsubsection*{The spin-diffusion equation}In the spin-diffusion equation
{\eqref{eq:spindifproc}}, $J \left( \grad \s, \m \right)$
is the spin current, $\Dzero\in L^\infty(\Om)$ is the diffusion coefficient,
$\lambda_{\tmop{sf}}\in\RR_+$ is the characteristic length for spin-flip relaxation,
and $\lambda_J\in\RR_+$ is related to the electron's mean free path. The spin current
is given by
$$
  J \left( \grad \s, \m \right) \assign 2 \Dzero
  \left[ \grad \s- \beta \beta' (\nabla \s \cdot
  \m_{}) \otimes \m \right] - \frac{\beta
  \mu_B}{e} \tmmathbf{j}_e \otimes \m, 
$$
where $\tmmathbf{j}_e$ is the applied electric current, $0 < \beta, \beta' <
1$ are the dimensionless spin-polarization parameters of the magnetic layers,
$\mu_B>0$ is the Bohr magneton, and $e>0$ is the charge of the electron. With given $\gamma\in\RR_+$, the
diffusion coefficient $\Dzero\in L^\infty(\Om)$ satisfies $\Dzero(x)\geqslant \gamma$ for almost all $x\in\Om$. The last term
in {\eqref{eq:spindifproc}} represents the interaction between the spin
accumulation and the background magnetization, and it is responsible for the
transfer of angular momentum between them.

\subsubsection*{The Landau--Lifshitz--Gilbert equation}In
{\eqref{eq:LLGphysics}}, $\gr\in\RR_+$ is the gyromagnetic ratio, $j_0$ models the
strength of the interaction between the spin and the magnetization, and
$\alpha$ is the dimensionless damping parameter. The first term on the
right-hand side of {\eqref{eq:LLGphysics}} describes a precession around the
field $\heff [\m] + j_0 \s$+$\tmmathbf{f}$, whereas the
(phenomenological) second term accounts for dissipation in the system. The time-dependent vector
field $\tmmathbf{f}$ is the so-called applied field, and it is assumed to be
unaffected by variations of $\m$. The effective field $\heff$
includes contributions originating from different spatial scales and is defined
by the negative first-order variation of the micromagnetic energy
functional $\mathcal{G}_{\Om}$, i.e.\
\begin{equation}
  \heff [\m] \assign - \partial_{\m} \mathcal{G}_{\Om}.
  \label{eq:heff}
\end{equation}
For single-crystal ferromagnets, the
micromagnetic energy functional reads as
({\tmabbr{cf.}}~{\cite{BrownB1963,hubert2008magnetic}})
\begin{equation}
  \mathcal{G}_{\Om} (\m) \assign \underset{= : \mathcal{E}_{\Om}
  (\m)}{\frac{c_{\tmop{ex}}}{2} \int_{\Om \; \;} \left| \grad
  \m \right|^2 } + \underset{= : \mathcal{A}_{\Om}
  (\m)}{\kappa \int_{\Om  \;} \phian (\m) }  \underset{= :
  \mathcal{W}_{\Om} (\m)}{- \frac{\mu_0}{2} \int_{\Om \; \;} \hd
  \left[ \m \chi_{\Om} \right] \cdot \m} 
  \label{eq:GLunorm}
\end{equation}
with $\m \in H^1 ( \Om, \Stwo^2 )$, and recall that $\m
\chi_{\Om}$ is the extension by zero of $\m$ to $\RR^3$. Combining
{\eqref{eq:heff}} and {\eqref{eq:GLunorm}}, we obtain the following expression
for the effective field:
\begin{equation}
  \heff [\m] = c_{\tmop{ex}} \Delta \m- \kappa \grad \phian
  (\m) + \mu_0 \hd \left[ \m \chi_{\Om} \right] 
  \label{eq:exprefffield} .
\end{equation}
In {\eqref{eq:GLunorm}}, the {\tmem{exchange energy}} $\mathcal{E}_{\Omega}$
penalizes nonuniformities in the orientation of the magnetization,  and $c_{\tmop{ex}} > 0$
is the exchange stiffness constant. The {\tmem{magnetocrystalline anisotropy
energy}} $\mathcal{A}_{\Omega}$ accounts for the existence of preferred
directions of the magnetization: its energy density $\phian : \Stwo^2
\rightarrow \RR^+$ vanishes only on a finite set of directions (the so-called
{\tmem{easy directions}}); $\kappa$ is the anisotropic constant. The third
term $\mathcal{W}_{\Omega}$ is the {\tmem{magnetostatic self-energy}}, i.e.,
the energy due to the demagnetizing field $\hd$ generated by $\m$.
The operator $\hd : \m \mapsto \hd [\m]$ provides, for every
$\m \in L^2 ( \RR^3, \RR^3 )$, the unique solution in
$L^2 ( \RR^3, \RR^3 )$ of the Maxwell--Amp{\`e}re equations of
magnetostatics
{\cite{BrownB1962,DiFratta2019,praetorius2004analysis}}:
\begin{equation}
  \left\{\begin{array}{l}
    \divv \tmmathbf{b} [\m] \eqs 0,\\
    \mathbf{curl}\,  \hd [\m] \eqs 0,\\
    \tmmathbf{b} [\m] = \mu_0 \left( \hd [\m]
    +\m \right),
  \end{array} \right. \label{eq:FarMaxdemag}
\end{equation}
where $\tmmathbf{b} [\m]$ denotes the magnetic flux density, and
$\mu_0$ is the magnetic permeability of vacuum.

\begin{remark}
  In what follows
  we assume that $\Om = \Omp$ (see, e.g., {\cite{MR2411058}} and
  {\cite{MR4021901}}). This permits
to simplify the notation and to state the weak-strong uniqueness result in a more elegant form. Everything we are
  going to prove still holds in the case $\Om \subset \Omp$. However, slight modifications may be required depending on the degree of
  smoothness one agrees to impose to a weak solution in order to elevate it to the rank of a
  strong solution ({\tmabbr{cf.}}~Lemma \ref{lemma:regularitysm0} and Remark
  \ref{rmk:onOmdiffOmp}).
\end{remark}

\subsection{Contributions of the present work: {\sdllg} in the limit of long
times}Already {\jname{Zhang}}, {\jname{Levy}}, and
{\jname{Fert}}~{\cite{zhang2002mechanisms}} observed that in the
non-ballistic regime, the time scales involved in {\sdllg}
{\eqref{eq:spindifproc}}--{\eqref{eq:LLGphysics}} are very different. For the
spin accumulation $\s$ the characteristic time scales are of the order of
$\lambda^2_{\tmop{sf}} / \left( 2 \Dzero \right)$ and
$\lambda^2_J / \left( 2 \Dzero \right)$. Given the typical spatial scales involved in spintronic
applications, these quantities are
of the order of picoseconds.
On the other hand, the characteristic time scale
for the magnetization dynamics depends on the inverse of the modulus of the precessional term $\gr \m_{\varepsilon} \times
( \heff [\m_{\varepsilon}] + j_0
\s_{\varepsilon} +\tmmathbf{f})$. In typical spintronic
applications, this time scale is of the order of nanoseconds. Therefore,
{\guillemotleft}{\tmem{as long as one is interested in the magnetization
process of the local moments}}, {\tmem{one can always treat the spin
accumulation in the limit of long
times}}{\guillemotright}~{\cite{zhang2002mechanisms}}. 

Formally, the observations in {\cite{zhang2002mechanisms}} can be
summarized in the claim that as long as the main interest is in magnetization
dynamics, one can forget about the term $\dt{\s}$ in \eqref{eq:spindifproc} and focus on the analysis of the {\sllg} equation
$$
  \dt{\m} \eqs - \gr \m \times \left( \heff
  [\m] + j_0 \s+\tmmathbf{f} \right) + \alpha \m
  \times \dt{\m} \quad \text{in } \Om \times \RR_+,
$$
with the spin accumulation $\s$ satisfying the {\tmem{steady-state}} equation
\begin{equation*}
  \divb \left[ J \left( \grad \s,
  \m \chi_{\Om} \right) \right] - \frac{2
  \Dzero}{\lambda^2_{\tmop{sf}}} \tmmathbf{s}- \frac{2 \Dzero}{\lambda^2_J}
  \s \times \m \chi_{\Om} \eqs \tmmathbf{0} \quad \text{in } \Om \times \RR_+ , 
\end{equation*}
One aim of this
paper is to turn this observation into a quantitative statement.

For that, one has to rescale the original domain $\Omega$ by a scaling factor $\ell$ to be of size one (without relabelling) and treat $\Dzero / \lambda^2_{\tmop{sf}}$, $\Dzero / \lambda^2_J$, $\mu_B\tmmathbf{j}_e/e$, and $\Dzero/\ell^2$  as quantities of the same order. Moreover, we also assume that $j_0$, $\tmmathbf{f}$, $c_{\tmop{ex}}/\ell^2$ are of the same order and $\gamma_0 c_{\tmop{ex}}/\ell^2 \ll \Dzero/\ell^2$.
We introduce the small parameter $\varepsilon = \gamma_0 c_{\tmop{ex}}/\Dzero$ and note that in typical spintronic applications $\varepsilon \ll 1$ is of the order $10^{-2}\ldots 10^{-4}$.

After rescaling the time of the system by the order of $\gamma_0 c_{\tmop{ex}}/\ell^2$ (see \cite[Section~2.2]{Abert2014} for a detailed analysis), one can rewrite the {\sdllg} equation \eqref{eq:spindifproc}--\eqref{eq:LLGphysics} as
\begin{eqnarray}
  \varepsilon \dt{\s_{\varepsilon}} & \eqs & \divb \left[
  \mathcal{J} \left( \grad \s_{\varepsilon},
  \m_{\varepsilon} \right) \right] - \gamma_1
  \s_{\varepsilon} - \gamma_2 \s_{\varepsilon} \times
  \m_{\varepsilon} \quad \text{in } \Om \times \RR_+, 
  \label{eq:systemLLg+spin2}\\
  \dt{\m_{\varepsilon}} & \eqs & -\m_{\varepsilon} \times
  \left( \heff [\m_{\varepsilon}] + j_0 \s_{\varepsilon}
  +\tmmathbf{f}_{\varepsilon} \right) + \alpha \m_{\varepsilon}
  \times \dt{\m_{\varepsilon}} \quad \text{in } \Om \times \RR_+, 
  \label{eq:systemLLg+spin2.2}
\end{eqnarray}
with the spin current given by
\begin{equation} 
  \mathcal{J} \left( \grad \s, \m \right) \assign
  \Dzero  \left[ \grad \s- \beta \beta' (\nabla \s \cdot
  \m_{}) \otimes \m \right] -
  \frac{\beta}{2} \tmmathbf{j}_e \otimes \m.  \label{eq:exprJval}
\end{equation}
Here, $\gamma_1$ and $\gamma_2$ are positive quantities of order one, while $\varepsilon \ll 1$ is a small parameter.
In \eqref{eq:systemLLg+spin2} and \eqref{eq:systemLLg+spin2.2}, to avoid introducing new notation, we denoted by the very same symbols the rescaled version of the physical quantities appearing in the {\sllg} equation  \eqref{eq:spindifproc}--\eqref{eq:LLGphysics}.

For the diffusion coefficient $\Dzero\in L^\infty(\Om)$ we assume the existence of a positive constant $\gamma\in\RR_+$ such that $\Dzero(x)\geqslant \gamma$ for almost all $x\in\Om$. However, without loss of generality in our arguments, we assume that $\gamma_1$, $\gamma_2$ are positive constants.   

We note that, formally, if $\s_\varepsilon\to \s$, $\m_\varepsilon\to \m$, and $\varepsilon \dt{\s_{\varepsilon}} \to 0$ then we recover the {\sllg} equation predicted  
in {\cite{zhang2002mechanisms}}, i.e., 
\begin{equation}
  \dt{\m \eqs  \m \times \left( \heff [\m] + j_0
  \s+\tmmathbf{f} \right) + \alpha \m \times
  \dt{\m} \quad \text{in } \Om \times \RR_+,} \label{eq:sLLg+spin1.1}
\end{equation}
with the spin accumulation $\s$ satisfying the {\tmem{steady-state}} equation
\begin{equation}
  \divb \left[ \mathcal{J} \left( \grad \s, \m \right)
  \right] - \gamma_1 \s- \gamma_2 \s \times \m
  \eqs \tmmathbf{0} \quad \text{in } \Om \times \RR_+ , \label{eq:sLLg+spin1.2}
\end{equation}
and the spin current $\mathcal{J}$ given by \eqref{eq:exprJval}.

The main aim of the paper is to show that this is indeed the case, i.e., that in the limit of long times, the 
rescaled {\sdllg}
\eqref{eq:systemLLg+spin2}--\eqref{eq:systemLLg+spin2.2} reduces to the simpler {\sllg} \eqref{eq:sLLg+spin1.1}--\eqref{eq:sLLg+spin1.2}, in which  {\acro{LLG}}  is coupled to the nonlinear (but
still nonlocal) steady-state equation. The result provides a solid
mathematical ground to the qualitative analysis of {\jname{Zhang}},
{\jname{Levy}}, and {\jname{Fert}} {\cite{zhang2002mechanisms}} in the context
of magnetic multilayers. Besides, the argument shows the existence (but possibly {\tmem{nonuniqueness}}) of global weak solutions of {\sllg}.

We complete the paper by proving the weak-strong uniqueness of the solutions of
the reduced {\sdllg}, i.e., weak solutions coincide with the (unique) strong solution as long as the latter exists in time.
Weak-strong uniqueness
results are of particular relevance in the numerical integration of
{\acro{LLG}} systems. Indeed, available unconditionally convergent integrators
assure that subsequences of the computed discrete solutions  converge weakly
towards a weak solution of {\acro{LLG}}. Weak-strong
uniqueness results guarantee that all these numerical schemes will converge
towards the same limit (even for the full sequence of computed solutions), at
least as long as a strong solution exists. We refer to
{\cite{MR2257110,MR2379897}} for some seminal works on the numerical analysis of
plain {\acro{LLG}} and to {\cite{Abert2014,dilinear}}
for the analysis of some coupled {\acro{LLG}} systems.  In \cite{abert2015three,ruggeri2016coupling}, the observations in {\cite{zhang2002mechanisms}} are empirically validated through a comparative numerical analysis of {\sdllg} and {\sllg} models: It is underlined how {\sllg} can be more effective in describing magnetization dynamics, since it leads to the same experimental results as for  {\sdllg}, but allows for much larger time steps of the numerical integrator.

\subsection{State of the art}Many research works contributed to the study of
solutions of {\acro{LLG}}. Existence and nonuniqueness of global weak
solutions to {\acro{LLG}} {\eqref{eq:LLGphysics}} date back to
{\cite{MR1167422}}. In {\cite{MR3999335}} (see also
{\cite{zamponi2016analysis}}), the Maxwell--Landau--Lifshitz equation coupled
with spin accumulation is considered, and a suitable regularization procedure
is used to obtain the existence of global weak solutions. For the system we
are interested in, the existence (and nonuniqueness) of global weak solutions
is proved in {\cite{garcia2007spin}} (see also {\cite{Abert2014}}).
{\sdllg} {\eqref{eq:systemLLg+spin2}}--{\eqref{eq:systemLLg+spin2.2}} is a
rescaled version of the three-dimensional model introduced in
{\cite{garcia2007spin}}.

The uniqueness of weak solutions depends on the regularity class they belong
to. Indeed, in the class of smooth functions, there exists at most one
solution of {\acro{LLG}} (see, e.g., {\cite{Carbou2001}}). Therefore, a
natural question is whether existing smooth and weak solutions coincide,
rather than coexist, i.e., whether a weak-strong uniqueness result holds for
{\acro{LLG}}. Such a question is ubiquitous in the analysis of {\acro{PDE}}s
since the positive answer given by {\jname{Leray}} for the Navier--Stokes
equations {\cite{MR3838050}}. For {\acro{LLG}}, weak-strong uniqueness has
only been investigated recently. In {\cite{dumas2014weak}}, weak-strong
uniqueness is proved in the simplified setting where $\Om = \RR^3$ (i.e., possible
boundary conditions are neglected) and $\heff$ consists only of the
leading-order exchange contribution. Despite the simplified setting, the
proof already involves much tedious algebra. In {\cite{di2020weak}}, weak-strong
uniqueness for the solutions of {\acro{LLG}} is obtained for the full relevant
$3 d$ setting. In particular, the analysis accounts also for the
Dzyaloshinskii--Moriya interaction, which is the primary mechanism behind the
emergence of magnetic skyrmions, as well as for the demagnetizing field $\hd$ from \eqref{eq:FarMaxdemag}. However, no coupling of {\acro{LLG}} with
other nonlinear {\acro{PDE}} systems is taken into account. In section
\ref{sec:wekastrong}, we show how the approximation argument in
{\cite{di2020weak}} can be used to derive weak-strong uniqueness of solutions
of the reduced {\sllg}.

From the above discussion, there emerges the need for regularity results for
{\acro{LLG}}. In that regard, we recall that {\acro{LLG}} is intimately
related to the harmonic map heat flow from $\Om$ into $\Stwo^2$ and 
stationary harmonic maps, for which one cannot expect to have general
regularity results in dimensions higher than two {\cite{Riviere1995}}. For
sufficiently small initial data, there exists a (unique) strong solution which
is global in time {\cite{Carbou2001,Feischl2017}}, whereas for general initial
data, even in $2 d$, solutions may develop finitely many point singularities
in finite time {\cite{Harpes2004}}. In {\cite{MR2411058}}, the authors prove the
existence of global smooth solutions of the spin-polarized transport equation
({\sdllg}) in $2 d$ for small initial data. 
For general dimensions, partial regularity
has been investigated in {\cite{Liu2004,Melcher2005,Wang2006}} for
{\acro{LLG}}, and, more recently, in {\cite{MR4021901}} for {\sdllg}.

\subsection{Outline}The paper is organized as follows. In
section~\ref{sec:mainresultsstated}, we state our main results:
Theorem~\ref{thm:Theorem1} concerns the analysis of {\sdllg} in the
limit of long times, whereas Theorem~\ref{thm:Theorem2} states the
weak-strong uniqueness result for the limiting equation. The proofs of
Theorems~\ref{thm:Theorem1} and \ref{thm:Theorem2} are given in 
sections~\ref{sec:proofthm1} and \ref{sec:wekastrong}, respectively.

\subsection{Notation} For $\Om \subset \RR^3$ open and bounded, we denote by
$H^1(\Om)^{\ast}$ the dual space of $H^1(\Om)$ and
by $\langle \cdot, \cdot \rangle_{\Om}$ the corresponding duality pairing,
understood in the sense of the Gelfand triple $H^1 ( \Om)
\subset L^2( \Om) \subset H^1 ( \Om )^{\ast}$. In
particular, if $u, v \in L^2 ( \Om )$, then $\langle u, v
\rangle_{\Om}$ denotes the usual inner product in $L^2(\Om)$ and
$\| \cdot \|_{\Om}$ is the induced $L^2 ( \Om )$ norm. Vector-valued functions are denoted in
boldface, but we do not embolden the function spaces they belong to; the
context will clarify what we mean. Instead, we use the symbol $H^1( \Om,
\Stwo^2)$ to denote the metric subspace of $H^1 ( \Om)$
consisting of $\Stwo^2$-valued functions. When dealing with time-dependent
vector fields $\tmmathbf{u}: \Om \times \RR_+ \rightarrow \RR^3$, we will
often use the symbol $\tmmathbf{u} (t)$ to denote the section $\tmmathbf{u}
(\cdot, t) : x \in \Om \mapsto \tmmathbf{u} (x, t) \in \RR^3$; again, the
context will clarify the meaning. Finally, for every $\Ti \in \RR_+$ we set
$\Om_{\Ti} \assign \Om \times \left( 0, \Ti \right)$.

\section{Statement of main results}\label{sec:mainresultsstated}

We recall the definition of a global weak solution of the {\sdllg} system
\eqref{eq:systemLLg+spin2}--\eqref{eq:systemLLg+spin2.2} as given in \cite{garcia2007spin, Abert2014}, where also existence results are shown. The definition naturally extends the
notion of global weak solution of {\acro{LLG}} introduced in {\cite{MR1167422}}. To simplify the notation, we neglect the contributions from
crystalline anisotropy and the external applied field $\tmmathbf{f}$. It is straightforward to include them in the analysis, and details are left to the reader.

\begin{definition}
   Let $\Om \subset \RR^3$ be a bounded domain, $\m^{\ast} \in H^1
  ( \Om, \Stwo^2 )$ and $\s^{\ast} \in H^1 ( \Om
  )$. For $\varepsilon > 0$, the pair
  $(\m_{\varepsilon}, \s_{\varepsilon}) \in L^{\infty}
  ( \RR_+, H^1 ( \Om, \Stwo^2 ) ) \times L^{\infty}
  ( \RR_+, L^2 ( \Om ) )$ is called a global weak
  solution of the {\rm {\sdllg}} system
  {\eqref{eq:systemLLg+spin2}}--{\eqref{eq:systemLLg+spin2.2}} if for every
  $\Ti > 0$ the following properties {\rm (i)--(v)} are satisfied:
  \begin{itemize}
    \item[{\rm(i)}] $\m_{\varepsilon} \in H^1 ( \Om_{\Ti}, \Stwo^2
    )$ and $\m_{\varepsilon} (0) =\m^{\ast}$ in the
    sense of traces;
    
    \item[{\rm(ii)}] $\s_{\varepsilon} \in L^2 ( \RR_+, H^1 ( \Om
    ) ), \partial_t \s_{\varepsilon} \in L^2 (
    \RR_+, H^1 ( \Om )^{\ast} )$, and
    $\s_{\varepsilon} (0) =\s^{\ast}$ in the sense of
    traces;
    
    \item[{\rm(iii)}] for every $\tmmathbf{\varphi} \in H^1 ( \Om_{\Ti} )$,
    there holds
    \begin{eqnarray}
      \int_0^{\Ti} \langle \partial_t \m_{\varepsilon},
      \tmmathbf{\varphi} \rangle_{\Om} & \eqs & \int_0^{\Ti} \alpha
      \langle \partial_t \m_{\varepsilon}, \tmmathbf{\varphi} \times
      \m_{\varepsilon} \rangle_{\Om} + c_{\tmop{ex}} \left\langle \grad
      \m_{\varepsilon}, \grad (\tmmathbf{\varphi} \times
      \m_{\varepsilon}) \right\rangle_{\Om}  \nonumber\\
      &  & \hspace{2em} - \mu_0\int_0^{\Ti} \left\langle \hd
      [\m_{\varepsilon}], \tmmathbf{\varphi} \times
      \m_{\varepsilon} \right\rangle_{\Om} - \int_0^{\Ti} j_0
      \langle \s_{\varepsilon}, \tmmathbf{\varphi} \times
      \m_{\varepsilon} \rangle_{\Om};  \label{eq:defweaksolornew0}
    \end{eqnarray}
    \item[{\rm(iv)}] for all $\tmmathbf{\varphi} \in L^2( \RR_+, H^1 ( \Om
    ) )$, there holds
    \begin{eqnarray}
      \varepsilon \int_0^{\Ti} \langle \partial_t \s_{\varepsilon},
      \tmmathbf{\varphi} \rangle_{\Om} & \eqs & - \int_0^{\Ti}  \langle \mathcal{J}
      \left( \grad \s_{\varepsilon}, \m_{\varepsilon}
      \right) , \grad \tmmathbf{\varphi}\rangle_\Om - \gamma_1 \int_0^{\Ti}
      \langle \s_{\varepsilon} , \tmmathbf{\varphi} \rangle_\Om \nonumber\\
      &  &  \hspace{2em} - \gamma_2
     \int_0^{\Ti} \langle \s_{\varepsilon} \times
      \m_{\varepsilon} , \tmmathbf{\varphi} \rangle_\Om - \frac{\beta}{2}
      \int_0^{\Ti} \int_{\partial \Om} (\m_{\varepsilon} \cdot
      \tmmathbf{\varphi}) (\tmmathbf{j}_e \cdot \tmmathbf{n}),
      \label{eq:weakseps}
    \end{eqnarray}
    where $\tmmathbf{j}_{e} \in L^2 ( \RR_+, H^1 (
    \Om ) )$ is a prescribed current and $\mathcal{J}$ is given by
    {\eqref{eq:exprJval}};

    \item[{\rm(v)}] the following energy inequality holds:
    \begin{subequations}
    \begin{equation}
      \mathcal{F}_{\Om} \left( \m_{\varepsilon} \left( \Ti \right)
      \right) + \alpha \int_0^{\Ti} \left\| \dt{\m}_{\varepsilon}
      \right\|^2_{\Om} \; \leqslant \; \mathcal{F}_{\Om} (\m^{\ast}) +
      \int_0^{\Ti} \langle \partial_t \m_{\varepsilon}, j_0
      \s_{\varepsilon} \rangle_{\Om},  \label{eq:eninequ0}
    \end{equation}
    with
    \begin{equation*}
      \mathcal{F}_{\Om} \left( \m_{\varepsilon} \left( \Ti \right)
      \right) \; \eqs \; \frac{c_{\tmop{ex}}}{2} \int_{\Om \; \;} \left| \grad
      \m_{\varepsilon} \right|^2 - \frac{\mu_0}{2} \int_{\Om} \hd
      \left[ \m_{\varepsilon} \chi_{\Om} \right] \cdot
      \m_{\varepsilon} . 
    \end{equation*}
    \end{subequations}
  \end{itemize}
\end{definition}

\begin{remark}
 If $(\m_{\varepsilon}, \s_{\varepsilon})$ is a weak solution of {\sdllg}, then standard results guarantee that $\m_{\varepsilon} \in C(0, \Ti ; L^2(\Om))$ and $\s_\varepsilon \in C(0, \Ti ; L^2(\Om))$, cf., e.g., \cite[Section 5.9.2, Theorems 2 and 3]{evans2010partial}.
\end{remark}

\begin{remark}
  Since the seminal paper {\cite{MR1167422}}, the definition of global
  weak solutions of {\acro{LLG}} includes the energy requirement
  (v), which even holds with equality for strong solutions
  of {\acro{LLG}}. However, in the definition of weak solutions of {\sdllg} given in
  {\cite{garcia2007spin}}, the condition (v) has been omitted. It was later shown in {\cite[Thm.~24]{Abert2014}} that this requirement can be satisfied. 
\end{remark}

\begin{remark}
  Throughout, $\grad \s=(\partial_i s_j)_{i,j=1}^3$ stands for the matrix whose columns are the
  gradients of the components of $\s$, i.e., the transposed of the
  Jacobian matrix of $\s$. In our notation, the Jacobian matrix is
  thus denoted by $\grad^{\T} \s$. The same remark applies to $\grad
  \m$. The vector $\divb$ operator acts on the columns of $\grad
  \s$: If $A : \Om \rightarrow \RR^{3
  \times 3}$ is a matrix-valued function, then $\divb A
  \eqs \sum_{i = 1}^3 \divv (A e_i) e_i$, with $(e_i)_{i = 1}^3$ the
  canonical basis of $\RR^3$, is the vector having
  for components the scalar divergence of the columns of $A$.
  In particular, $\divb  \grad \s= \lapl
  \s$. Also, if $\tmmathbf{\varphi}: \Om \rightarrow \RR^3$ is a
  vector field, then 
  \begin{equation}
    \divb A \cdot \tmmathbf{\varphi} \eqs \sum_{i = 1}^3  \varphi_i  \divv (A e_i) \eqs \sum_{i = 1}^3 \left[ \divv (\varphi_i A_{} e_i) - \grad
    \varphi_i \cdot A_{} e_i \right] \eqs \divv (A\tmmathbf{\varphi}) - A \Fsp \grad
    \tmmathbf{\varphi}. \label{eq:intbyparts}
  \end{equation}
  According to the divergence theorem, this leads to
  \begin{equation}
    \int_{\Om} \divb A \cdot \tmmathbf{\varphi}  \eqs \int_{\partial
    \Om} A\tmmathbf{\varphi} \cdot \tmmathbf{n}  - \int_{\Om} A
    \Fsp \grad \tmmathbf{\varphi}. \label{eq:divthmvec}
  \end{equation}
  From the previous considerations, it is clear that {\eqref{eq:weakseps}} is
  the natural weak formulation of~{\eqref{eq:systemLLg+spin2}}. Indeed,
  applying {\eqref{eq:divthmvec}} and {\eqref{eq:intbyparts}} to
  {\eqref{eq:exprJval}}, we have
  \begin{equation*}
    \int_{\Om} \divb  [(\nabla \s \cdot \m_{}) \otimes
    \m] \cdot \tmmathbf{\varphi}   \eqs  - \int_{\Om}
    (\nabla \s \cdot \m_{}) \otimes \m \Fsp
    \grad \varphi + \int_{\partial \Om} ((\nabla \s \cdot
    \m) \otimes \m )\tmmathbf{\varphi} \cdot
    \tmmathbf{n} 
  \end{equation*}
  and the last integrand gives 
\begin{equation*}
 ((\nabla \s \cdot
    \m) \otimes \m )\tmmathbf{\varphi} \cdot
    \tmmathbf{n} \;\eqs \;(\m \cdot \tmmathbf{\varphi}) 
  (\nabla \s \cdot \m_{}) \cdot \tmmathbf{n} \;\eqs\; (\m
  \cdot \tmmathbf{\varphi})  (\partial_{\tmmathbf{n}} \s \cdot
  \m_{}) = 0\quad \text{on } \partial \Om \, .
\end{equation*}  
  Similarly, integration by parts of
  the term $\divb (\tmmathbf{j}_e \otimes \m) \cdot
  \tmmathbf{\varphi}$ gives the boundary term
  \begin{equation*}
  (\tmmathbf{j}_e \otimes
  \m) \tmmathbf{\varphi} \cdot \tmmathbf{n}= (\m \cdot
  \tmmathbf{\varphi}) (\tmmathbf{j}_e \cdot \tmmathbf{n}).
  \end{equation*}
\end{remark}

For future reference, we collect here the mathematical assumptions on the physical parameters of the {\rm \sdllg} system
{\tmem{{\eqref{eq:defweaksolornew0}}--{\eqref{eq:weakseps}}}} that will be assumed throughout the paper:

\begin{hypotheses}
\Hypo[hyp:parameters]{%
\emph{Assumptions on the physical parameters of the system.} In what follows, we assume that  $\alpha,\cex,j_0,\mu_0$ are positive constants (cf.~\eqref{eq:defweaksolornew0}).
Also, we assume that $\Dzero\in L^\infty(\Om)$ and that there exists a positive constant $\gamma\in\RR_+$ such that $\Dzero(x)\geqslant \gamma$ for almost all $x\in\Om$. The coefficients $\gamma_1$, $\gamma_2$ are assumed to be  positive constants (cf.~\eqref{eq:weakseps}). Finally, we assume that $0<\beta,\beta'<1$ (cf.~\eqref{eq:exprJval}).}
\end{hypotheses}

Our first contribution is the following result concerning the behavior of the
spin transport equation in the limit $\varepsilon \rightarrow 0$. For the sake of clarity, we will often write the equations in strong form, although their weak counterpart is meant.
\begin{theorem}
  \label{thm:Theorem1}
  Let $\Om \subset \RR^3$ be a bounded Lipschitz domain, and assume \tmem{\ref{hyp:parameters}}.
	For every $\varepsilon > 0$, let $(\m_{\varepsilon},
  \s_{\varepsilon}) \in L^{\infty} ( \RR_+, H^1 ( \Om,
  \Stwo^2 ) ) \times L^{\infty} ( \RR_+, L^2 ( \Om
  ) )$ be a weak solution of the {\rm \sdllg} system
  {\tmem{{\eqref{eq:defweaksolornew0}}--{\eqref{eq:weakseps}}}}. Then, there
  exists a vector field $\m_0 \in L^2 ( \RR_+ , H^1 ( \Om, \Stwo^2
  ) )$ such that
  \begin{eqnarray}
    \m_{\varepsilon} \rightharpoonup \m_0 &  &
    \text{{\tmem{weakly in}} } L^2_{\tmop{loc}} ( \RR_+, H^1 ( \Om,
    \Stwo^2 ) ) , \nonumber
  \end{eqnarray}
  and $\m_0\in H^1 ( \Om_\Ti,
    \Stwo^2 ) $ for every $\Ti>0$.
Moreover, the vector field $\m_0$ satisfies the {\rm \sllg} equation
  \begin{equation}
    \partial_t \m_0 \eqs -\m_0 \times
     (\heff [\m_0] + j_0 \mathcal{H}_s
    [\m_0] +\tmmathbf{f}) 
    + \alpha \m_0 \times \partial_t
    \m_0 \quad \text{{\tmem{in}} }  \Om \times \RR_+,
    \label{eq:CauchyequsValstationary}
  \end{equation}
  with
  \[ \left\{ \begin{array}{lll}
       \m_0 (0) \eqs \m^{\ast} (x) & \text{{\tmem{in}}} &
       \Om,\\
       \partial_{\tmmathbf{n}} \m_0 \eqs 0 & \text{{\tmem{on}}} &
       \partial \Om \times \RR_+ .
     \end{array} \right. \]
  Here, $\mathcal{H}_s : \m \in H^1 ( \Om, \Stwo^2 )
  \mapsto \mathcal{H}_s [\m] \in H^1 ( \Om, \RR^3 )$ denotes
  the nonlinear operator which maps every $\m \in H^1 ( \Om,
  \Stwo^2 )$ to the unique solution $\s \assign
  \mathcal{H}_s [\m] \in H^1 ( \Om, \RR^3 )$ of the
  stationary spin-diffusion equation
  \begin{equation}
    - \divb \left[ \mathcal{J} \left( \grad \s, \m \right)
    \right] + \gamma_1 \s+ \gamma_2 \s \times \m
    \eqs 0 \text{ \tmem{in }} \Om, \quad \text{\tmem{subject to }} \partial_{\tmmathbf{n}} \s \eqs
  0 \text{\tmem{ on }} \partial \Om. \label{eq:stspindiff}
  \end{equation}
\end{theorem}

We have written the reduced equations
{\eqref{eq:CauchyequsValstationary}}--{\eqref{eq:stspindiff}} in strong form to improve the readability (their weak formulation is immediate to derive). As for
the system {\eqref{eq:systemLLg+spin2}}--{\eqref{eq:systemLLg+spin2.2}},
uniqueness of weak solutions is, in general, out of the question. In fact, when
$\tmmathbf{j}_e \equiv 0$, the system {\sllg} reduces to the classical LLG
equation for which possible nonuniqueness has been shown in {\cite{MR1167422}}.

In the statement of Theorem~\ref{thm:Theorem1}, we assumed that $\Om$ is a Lipschitz domain. Our second contribution is the following weak-strong uniqueness result, for which higher regularity of the domain $\Om$, as well as regularity assumptions on the diffusion coefficient $D_0$ become essential.

\begin{theorem}
  \label{thm:Theorem2}
Let $\Om \subset \RR^3$ be a bounded domain with a smooth boundary, and assume that \tmem{\ref{hyp:parameters}} holds with $D_0\in C^{\infty} (
  \overline{\Om})$.
  Let $\m^{\ast} \in C^{\infty} (
  \overline{\Om}, \Stwo^2 )$ and $\Ti > 0$. Suppose that
  $\m_1, \m_2 \in L^{\infty} ( \RR_+ , H^1 (
  \Om, \Stwo^2 ) )$ are two global weak solutions of
  {\tmem{{\eqref{eq:CauchyequsValstationary}}}}. If $\m_1 \in
  C^{\infty} ( \overline{\Om_{\Ti}}, \Stwo^2 )$, then
  \[ \m_1 \equiv \m_2 \quad \text{a.e.~in } \Om_{\Ti} . \]
\end{theorem}

\begin{remark}
  A closer look to the proof of Theorem~\ref{thm:Theorem2} shows that it is
  sufficient to assume $\m^{\ast} \in C^3 ( \overline{\Om},
  \Stwo^2 )$, $D_0\in C^{3}( \overline{\Om})$, and $\m_1 \in C^3 ( \overline{\Om_{\Ti}},
  \Stwo^2 )$, but we do not dwell on this. Also, we provide a proof of the
  weak-strong uniqueness result in a more general form (see section \ref{sec:wsgensett}).
\end{remark}

In the proofs of Theorems
\ref{thm:Theorem1} and \ref{thm:Theorem2}, we will make use of some properties of the demagnetizing field operator $\hd$ that we recall here ({\tmabbr{cf.~}}{\cite{DiFratta2019}}). If $\Om$ is a
bounded domain, $\m \in L^2( \Om, \RR^3)$, and $\hd[ \m
\chi_{\Om}] \in L^2( \RR^3, \RR^3)$ is a solution of the
Maxwell--Amp{\`e}re equations {\eqref{eq:FarMaxdemag}} then, by Poincar{\'e}'s
lemma, $\hd[ \m \chi_{\Om}] = \nabla v_{\m}$, where $v_m$ is the
unique solution in $H^1 ( \RR^3)$ of the Poisson's equation
\begin{equation}
  - \Delta v_{\m} = \, \tmop{div} \left( \m \chi_{\Om} \right)
	\quad\mbox{in }\RR^3. \label{eq:potential}
\end{equation}
Therefore, the demagnetizing field can be described as the map which to every
magnetization $\m \in L^2( \RR^3, \RR^3)$ associates the
distributional gradient of the unique solution of {\eqref{eq:potential}} in
$H^1(\RR^3)$. It is easily seen that the map $- \hd : \m \in L^2
( \Om, \RR^3 ) \mapsto - \nabla v_{\m} \in L^2( \RR^3, \RR^3)$ defines a self-adjoint and positive-definite bounded linear operator
from $L^2( \RR^3, \RR^3)$ into itself:
\begin{equation}
  - \int_{\Om} \hd [ \m_1 \chi_{\Om} ] \cdot \m_2 
	= - \int_{\Om} \hd [\m_2 \chi_{\Om} ] \cdot \m_1 \label{eq:hdself}
\end{equation}
and
\begin{equation}
  - \int_{\Om} \hd [ \m_1 \chi_{\Om} ] \cdot \m_1 = \int_{\RR^3} \left|
  \hd [ \m_1 \chi_{\Om} ] \right|^2 \leqslant \int_{\Om} \left| \m_1
  \right|^2 \label{eq:conthd}
\end{equation}
for every $\m_1, \m_2 \in L^2( \Om, \RR^3)$.


\section{From {\sdllg} to {\sllg}: Proof of Theorem
\ref{thm:Theorem1}}\label{sec:proofthm1}
For convenience, we split the proof in three steps.

{\noindent}{\tmstrong{Step 1 (Uniform estimates).}} For $\Ti > 0$, we test
{\eqref{eq:weakseps}} against $\tmmathbf{\varphi}=\s_{\varepsilon}$
to obtain
\begin{eqnarray}
  \varepsilon \int_0^{\Ti} \langle \partial_t \s_{\varepsilon},
  \s_{\varepsilon} \rangle & \eqs & - \int_0^{\Ti} \int_{\Om}
  \mathcal{J} \left( \grad \s_{\varepsilon},
  \m_{\varepsilon} \right) \Fsp \grad \s_{\varepsilon} -
  \gamma_1 \int_0^{\Ti} \| \s_{\varepsilon} \|_{\Omega}^2 \nonumber\\
  & & \qquad\qquad\qquad\qquad\qquad\qquad -\frac{\beta}{2} \int_0^{\Ti} \int_{\partial \Om} (\m_{\varepsilon}
  \cdot \s_{\varepsilon}) (\tmmathbf{j}_e \cdot \tmmathbf{n})
  \label{eq:est1} .
\end{eqnarray}
An integration by parts gives
\begin{equation*}
  \int_0^{\Ti} \langle \partial_t \s_{\varepsilon},
  \s_{\varepsilon} \rangle \eqs \frac{1}{2} \int_0^{\Ti} \partial_t
  \| \s_{\varepsilon} \|_{\Om}^2 \eqs \frac{1}{2} \|
  \s_{\varepsilon} (T) \|_{\Om}^2 - \frac{1}{2} \|
  \s^{\ast} \|_{\Om}^2 .
\end{equation*}
On the other hand,
\begin{eqnarray}
  \int_{\Om} \mathcal{J} \left( \grad \s_{\varepsilon},
  \m_{\varepsilon} \right) \Fsp \grad \s_{\varepsilon} &
  \eqs & \int_{\Om} D_0 \left| \grad \s_{\varepsilon} \right|^2 -
  D_0 \beta \beta' \m_{\varepsilon} \otimes \left( \grad
  \s_{\varepsilon} \cdot \m_{\varepsilon} \right) \Fsp
  \grad \s_{\varepsilon} 
  \nonumber\\
  & & \qquad \qquad \qquad \qquad\qquad\qquad \qquad- \frac{\beta}{2}  (\tmmathbf{j}_e \otimes
  \m_{\varepsilon}) \Fsp \grad \s_{\varepsilon}  \nonumber\\
  & \eqs & \int_{\Om} D_0 \left( \left| \grad \s_{\varepsilon}
  \right|^2 - \beta \beta' \left| \grad \s_{\varepsilon} \cdot
  \m_{\varepsilon} \right|^2 \right) - \frac{\beta}{2} 
  (\tmmathbf{j}_e \otimes \m_{\varepsilon}) \Fsp \grad
  \s_{\varepsilon} . \label{eq:est2}
\end{eqnarray}
Recall that $\beta \beta'<1$ (see {\cite{zhang2002mechanisms}}). Let $\gamma_{\beta} \assign (1 - \beta \beta')>0$ and $\gamma\assign \essinf_\Om \Dzero>0$.
With
\begin{equation}
  \left| \grad \s_{\varepsilon} \right|^2 - \beta \beta' \left|
  \grad \s_{\varepsilon} \cdot \m_{\varepsilon} \right|^2
  \geqslant \gamma_{\beta} \left| \grad \s_{\varepsilon} \right|^2,
\end{equation}
the estimates {\eqref{eq:est1}} and
{\eqref{eq:est2}} lead to
\begin{align}
  & \frac{\varepsilon}{2} \left\| \s_{\varepsilon} \left( \Ti \right)
  \right\|_{\Om}^2 + \gamma_1 \int_0^{\Ti} \| \s_{\varepsilon}
  \|_{\Om}^2 + \gamma \gamma_{\beta} \int_0^{\Ti} \left\| \grad
  \s_{\varepsilon} \right\|^2_{\Om} \notag \\
  & \qquad\qquad\leqslant 
  \frac{\varepsilon}{2} \| \s^{\ast} \|_{\Om}^2 + \frac{\beta}{2}
  \int_0^{\Ti}  \int_{\Om} (\tmmathbf{j}_e \otimes \m_{\varepsilon})
  \Fsp \grad \s_{\varepsilon}  - \frac{\beta}{2} \int_0^{\Ti} \int_{\partial \Om}
  (\m_{\varepsilon} \cdot \s_{\varepsilon})
  (\tmmathbf{j}_e \cdot \tmmathbf{n}) \notag \\
    & \qquad\qquad \leqslant  \frac{\varepsilon}{2} \| \s^{\ast} \|_{\Om}^2 +
  \frac{\beta}{2} \int_0^{\Ti}  \| \tmmathbf{j}_e \|_{\Om}  \left\| \grad
  \s_{\varepsilon} \right\|_{\Om} - \frac{\beta}{2} \int_0^{\Ti} \int_{\partial \Om}
  (\m_{\varepsilon} \cdot \s_{\varepsilon})
  (\tmmathbf{j}_e \cdot \tmmathbf{n}), 
\end{align}
The continuous embedding of $H^1(\Om)$ into $L^2
( \partial \Om)$ implies the existence of $\delta > 0$
such that (we use Young's inequality)
\begin{equation}
  \varepsilon \left\| \s_{\varepsilon} \left( \Ti \right)
  \right\|_{\Om}^2 + \gamma_1 \| \s_{\varepsilon} \|_{\Om_T}^2 +
  c_{\delta} \int_0^{\Ti} \left\| \grad \s_{\varepsilon}
  \right\|^2_{\Om} \leqslant \varepsilon \| \s^{\ast} \|_{\Omega}^2
  + 
  \delta \left( \|
  \tmmathbf{j}_e \|_{\Om \times \RR_+^{}}^2 + \| \tmmathbf{j}_e \|_{\partial
  \Om \times \RR_+}^2 \right) 
  \label{eq:boundforsT}
\end{equation}
for some constant $c_{\delta} > 0$ which depends only on $\delta,\beta, \gamma,
\gamma_{\beta}$, and $\Om$.
Taking the supremum over $T>0$, we infer that
\begin{equation}
  \gamma_1 \| \s_{\varepsilon} \|_{\Om \times \RR_+}^2 + c_{\delta}
  \int_{\RR_+} \left\| \grad \s_{\varepsilon} \right\|^2_{\Om}
  \leqslant \varepsilon \| \s^{\ast} \|_{\Om}^2 + \delta \left( \|
  \tmmathbf{j}_e \|_{\Om \times \RR_+^{}}^2 + \| \tmmathbf{j}_e \|_{\partial
  \Om \times \RR_+}^2 \right) . \label{eq:unifbounds}
\end{equation}
Therefore, $(\s_{\varepsilon})$ is {\tmem{uniformly bounded}} in
$L^2 ( \RR_+, H^1 ( \Om ) )$.

{\noindent}{\tmstrong{Step 2 (The steady-state limit).}} The uniform
bound on $(\s_{\varepsilon})$ implies the existence of a
(not relabeled) subsequence $(\s_{\varepsilon})$ in $L^2 (
\RR_+ ; H^1 ( \Om ) )$ such that $\s_{\varepsilon}
\rightharpoonup \s_0$ weakly in $L^2 ( \RR_+, H^1 ( \Om
) )$ as $\varepsilon\to 0$. In particular, for every $\Ti > 0$, we have
$\s_{\varepsilon} \rightharpoonup \s_0$ weakly in $L^2
( \Om_{\Ti}, \RR^3 )$  as $\varepsilon\to 0$. Next, for every $\tmmathbf{\varphi} \in
C^{\infty} ( \overline{\Om_{\Ti}}, \RR^3 )$, equation
{\eqref{eq:weakseps}} reads, in expanded form, as
  \begin{align}
    & \int_{\Om_T} \mathcal{J} \left( \grad \s_{\varepsilon},
    \m_{\varepsilon} \right) \Fsp \grad \tmmathbf{\varphi}+ \gamma_1
    \int_{\Om_T} \s_{\varepsilon} \cdot \tmmathbf{\varphi}+ \gamma_2
    \int_{\Om_T} (\s_{\varepsilon} \times
    \m_{\varepsilon}) \cdot \tmmathbf{\varphi}\quad \nonumber\\
    & \quad + \frac{\beta}{2}
    \int_0^{\Ti} \int_{\partial \Om} (\m_{\varepsilon} \cdot
    \tmmathbf{\varphi}) (\tmmathbf{j}_e \cdot \tmmathbf{n}) \overset{\eqref{eq:weakseps}}{\eqs}
    \varepsilon
    \int_0^T \langle \s_{\varepsilon}, \partial_t \tmmathbf{\varphi}
    \rangle - \varepsilon \Big[ \langle
    \s_{\varepsilon} (T), \tmmathbf{\varphi} (T) \rangle - \langle
    \s^{\ast}, \tmmathbf{\varphi} (0) \rangle \Big] \label{eq:tempSSL}
  \end{align}
with
\begin{align}
   \int_{\Om} \mathcal{J} \left( \grad \s_{\varepsilon},
  \m_{\varepsilon} \right) \Fsp \grad \tmmathbf{\varphi}  \overset{\eqref{eq:exprJval}}{\eqs}  &
  \int_{\Om} D_0 \left( \grad \s_{\varepsilon} \Fsp \grad
  \tmmathbf{\varphi} -   \beta \beta'   \left(
  \grad \s_{\varepsilon} \cdot \m_{\varepsilon} \right) \otimes \m_{\varepsilon} :
  \grad \tmmathbf{\varphi} \right) \nonumber \\
  & \qquad\qquad\qquad\qquad\qquad\qquad- \frac{\beta}{2} \int_{\Om} (\tmmathbf{j}_e \otimes
  \m_{\varepsilon}) \Fsp \grad \tmmathbf{\varphi}. 
\end{align}
Now, we use the energy inequality {\eqref{eq:eninequ0}}. By Young's inequality, we can absorb a part of $| \langle \partial_t
\m_{\varepsilon}, \s_{\varepsilon} \rangle |$ into the
left-hand side of {\eqref{eq:eninequ0}}. 
With positive constants $\alpha_0, \alpha_1 > 0$, we find that
\begin{eqnarray}
  \mathcal{F}_{\Om} \left( \m_{\varepsilon} \left( \Ti \right)
  \right) + \alpha_0 \int_0^{\Ti} \left\| \dt{\m}_{\varepsilon}
  \right\|^2_{\Om} & \overset{\eqref{eq:eninequ0}}{\leqslant} & \mathcal{F}_{\Om} (\m^{\ast}) +
  \alpha_1 \| \s_{\varepsilon} \|_{\Om \times \RR_+}^2 \, .
\end{eqnarray}
Thus, from the
uniform bound \eqref{eq:unifbounds} on $\| \s_{\varepsilon} \|_{\Om \times \RR_+}^2$ and \eqref{eq:conthd}, we infer a uniform bound on the family $\left(\m_{\varepsilon}\right)$ in 
$L^{\infty} ( \RR_+ , H^1 ( \Om,\Stwo^2 ) )$, as well as a uniform bound on $\left(
\dt{} \m_{\varepsilon} \right)$ in $L^2 ( \RR_+, L^2 ( \Om ) )$. 
By the Aubin--Lions--Simon lemma,
there hence exists $\m_0 \in L^{\infty} ( \RR_+ , H^1 ( \Om,
\Stwo^2 ) )$ such that, up to a subsequence,
\begin{eqnarray}
  \m_{\varepsilon} \;_{} \rightarrow & \m_0 & \quad
  \text{strongly in } C^0 ( 0, T ; L^2 ( \Om, \Stwo^2 )
  ),  \label{eq:conv1meps}\\
  \dt{\m_{\varepsilon}} \; \rightharpoonup & \dt{\m_0} &
  \quad \text{weakly in } L^2 ( \RR_+, L^2 ( \Om ) ) . 
  \label{eq:conv2meps}
\end{eqnarray}
In particular, $\m_{\varepsilon} \rightarrow \m_0$
strongly in $L^2 ( \Om_T )$, from which it follows that
$\m_{\varepsilon} \otimes \m_{\varepsilon} \rightarrow
\m_0 \otimes \m_0$ strongly in $L^2 ( \Om_T )$.
Indeed, $|\m_{\varepsilon}|=1=|\m_0|$ a.e. in $\Om_\Ti$ guarantees that
\begin{eqnarray}
  | \m_{\varepsilon} \otimes \m_{\varepsilon}
  -\m_0 \otimes \m_0 | & \leqslant & |
  \m_{\varepsilon} \otimes (\m_{\varepsilon}
  -\m_0) | + | (\m_{\varepsilon} -\m_0) \otimes
  \m_0 | \nonumber\\
  & \leqslant & 2 | \m_{\varepsilon} -\m_0 | . 
\end{eqnarray}
Hence, for every $\tmmathbf{\varphi} \in C^{\infty} (
\overline{\Omega_T}, \RR^3)$, we have
\begin{equation}
  \int_0^{\Ti} \left( \mathcal{J} \left( \grad \s_{\varepsilon},
  \m_{\varepsilon} \right), \grad \tmmathbf{\varphi} \right)_{\Om}
  \rightarrow \int_0^{\Ti} \left( \mathcal{J} \left( \grad \s_0,
  \m_0 \right), \grad \tmmathbf{\varphi} \right)_{\Om} .
\end{equation}
Overall, using that $\m_{\varepsilon} \rightharpoonup
\m_0$ weakly in $L^2 ( 0, \Ti ; L^2 ( \partial \Om )
)$, we obtain
\begin{eqnarray}
  - \lim_{\varepsilon \rightarrow 0} \varepsilon \left(
  \s_{\varepsilon} \left( \Ti \right), \tmmathbf{\varphi} \left( \Ti
  \right) \right)_{\Om} &\overset{\eqref{eq:tempSSL}}{\eqs} & \int_0^{\Ti} \left( \mathcal{J} \left( \grad
  \s_0, \m_0 \right), \grad \tmmathbf{\varphi}
  \right)_{\Om} + \int_0^{\Ti} (\gamma_1 \s_0 + \gamma_2
  \s_0 \times \m_0, \tmmathbf{\varphi})_{\Om} \nonumber\\
  &  & \qquad\qquad \qquad \qquad + \frac{\beta}{2}
  \int_0^{\Ti} \int_{\partial \Om} (\m_0 \cdot \tmmathbf{\varphi})
  (\tmmathbf{j}_e \cdot \tmmathbf{n}) , \label{eq:tocomplim}
\end{eqnarray}
since $\lim_{\varepsilon \rightarrow 0}  \int_0^{\Ti} \langle
\s_{\varepsilon}, \partial_t \tmmathbf{\varphi} \rangle =
\int_0^{\Ti} \langle \s_0, \partial_t \tmmathbf{\varphi} \rangle$.

It remains to compute the limit on the left-hand side of
{\eqref{eq:tocomplim}}. 
To this end, we observe that 
{\eqref{eq:boundforsT}} leads to
$\left\| \s_{\varepsilon} \left( \Ti
\right) \right\|_{\Om}^2 \leqslant \| \s^{\ast} \|_{\Om}^2 + \varepsilon^{-1}\delta
( \|
  \tmmathbf{j}_e \|_{\Om \times \RR_+^{}}^2 + \| \tmmathbf{j}_e \|_{\partial
  \Om \times \RR_+}^2 )$. This gives the estimate
\begin{equation}
\varepsilon \left\|
\s_{\varepsilon} \left( \Ti \right) \right\|_{\Om} \leqslant
\left[\varepsilon^2 \| \s^{\ast} \|_{\Om}^2 + \varepsilon \delta \left( \|
  \tmmathbf{j}_e \|_{\Om \times \RR_+^{}}^2 + \| \tmmathbf{j}_e \|_{\partial
  \Om \times \RR_+}^2 \right) \right]^{\frac{1}{2}},
\end{equation}
from which it follows that
\begin{equation}
  \varepsilon \left| \left( \s_{\varepsilon} \left( \Ti \right),
  \tmmathbf{\varphi} \left( \Ti \right) \right)_{\Om} \right| \leqslant 
  \varepsilon \left\| \s_{\varepsilon} \left( \Ti \right)
  \right\|_{\Om} \left\| \tmmathbf{\varphi} \left( \Ti \right) \right\|_{\Om}
  \rightarrow 0. \nonumber
\end{equation}
Summarizing, for every $\tmmathbf{\varphi} \in C^{\infty} (
\overline{\Om_{\Ti}}, \RR^3)$, the family $(\s_{\varepsilon},
\m_{\varepsilon})_{\varepsilon \in \RR_+}$ converges, for
$\varepsilon \rightarrow 0$, to a solution of the equation
\begin{align}
  \int_0^{\Ti} \left( \mathcal{J} \left( \grad \s_0, \m_0
  \right), \grad \tmmathbf{\varphi} \right)_{\Om} & + \int_0^{\Ti} (\gamma_1
  \s_0 + \gamma_2 \s_0 \times \m_0,
  \tmmathbf{\varphi})_{\Om} \nonumber\\ 
  & \qquad\qquad\qquad +\frac{\beta}{2} \int_0^{\Ti} \int_{\partial \Om}
  (\m_0 \cdot \tmmathbf{\varphi}) (\tmmathbf{j}_e \cdot \tmmathbf{n}) = 0. 
\end{align}
By density, the previous relation holds for every $\tmmathbf{\varphi} \in L^2
( \RR_+, H^1 ( \Om) )$. This gives the limit equation
{\eqref{eq:stspindiff}}. 

Finally, {\eqref{eq:CauchyequsValstationary}} follows by a standard application of the convergence relations 
\eqref{eq:conv1meps} and \eqref{eq:conv2meps} to the weak formulation of \acro{LLG} given in \eqref{eq:defweaksolornew0}.

\medskip

{\noindent}{\tmstrong{Step 3 (Unique solvability of the limit spin diffusion equation
\eqref{eq:stspindiff}.} The proof is completed as soon as we show
that for every $\m \in H^1( \Om, \Stwo^2)$ there exists
a {\tmem{unique}} solution $\s \assign \mathcal{H}_s [\m]$
of the stationary spin-diffusion equation {\eqref{eq:stspindiff}}. This is the
content of the next lemma, which also provides some details on the regularity
of the operator $\mathcal{H}_s$ that is exploited in the proof of the
weak-strong uniqueness theorem.

\begin{lemma} Let $\Om \subset \RR^3$ be a bounded Lipschitz domain.
  \label{lemma:regularitysm0} For any $\m \in H^1 ( \Om, \Stwo^2
  )$ and $\mathcal{J}$ given by {\tmem{{\eqref{eq:exprJval}}}}, there exists a unique solution $\s \assign \mathcal{H}_s
  [\m]_{} \in H^1 ( \Om, \RR^3 )$ of the stationary
  spin-diffusion equation
  \begin{equation}
    - \divb \left[ \mathcal{J} \left( \grad \s, \m \right)
    \right] + \gamma_1 \s+ \gamma_2 \s \times \m
    \eqs 0 \text{ {\tmem{in}} } \Om,  \quad \text{\tmem{subject to }} \partial_{\tmmathbf{n}} \s \eqs
  0 \text{\tmem{ on }} \partial \Om.\label{eq:stspindifftemp}
  \end{equation}
  Moreover, the
  operator $\mathcal{H}_s : H^1 ( \Om, \Stwo^2 ) \rightarrow H^1
  ( \Om, \RR^3 )$ maps the space $C^{k + 1} ( \overline{\Om}, \Stwo^2
  )$ into the space $C^k ( \bar{\Omega}, \RR^3 )$ provided that
  $D_0,\tmmathbf{j}_e \in C^{k + 1} ( \overline{\Om} )$ and $\Om$ is of
  class $C^{k + 1, 1}$.
\end{lemma}
\begin{remark}
  \label{rmk:onOmdiffOmp}Lemma \ref{lemma:regularitysm0} is the only point
  where a difference arise if one assumes that $\Om$ is strictly included in
  $\Omp .$ In this case, since $\m \equiv 0$ in $\Omp \setminus
  \Om$, a similar result on the regularity in $\Om$ (but up to the boundary
  $\partial \Om$) of the solutions of {\eqref{eq:stspindifftemp}} cannot be
  inferred due to the jump discontinuity of the $\m$-dependent
  coefficients. This is the reason why, when dealing with partial regularity
  results for weak solutions of the {\sdllg} equation, one assumes that $\Om =
  \Omp$ ({\tmabbr{cf.}}~{\cite{MR2411058,MR4021901}}). That said, everything
  we state still works in the case $\Om \subset \Omp$ as soon as one agrees
  that strong solutions have a smooth induced
  spin accumulation on the interface $\partial \Om \cap \Omp$.
\end{remark}

\begin{proof}
  First, we note that the spin-diffusion equation can be rearranged in a more
  convenient form. The weak formulation of {\eqref{eq:stspindiff}} gives the
  relation
  \begin{align}
    \int_{\Om} D_0  \left[ \grad \s- \beta \beta' (\nabla
    \s \cdot \m_{}) \otimes
    \m \right]  \Fsp  \grad \tmmathbf{\varphi} &+
    \int_{\Om} (\gamma_1 \s+ \gamma_2 \s \times
    \m) \cdot \tmmathbf{\varphi} \nonumber \\
    & \qquad\qquad\qquad\eqs  - \frac{\beta}{2}
    \left\langle \divb (\tmmathbf{j}_e \otimes \m),
    \tmmathbf{\varphi} \right\rangle_{\Omega} \label{eq:toapplyLaxMilg}
  \end{align}
  for every $\tmmathbf{\varphi} \in H^1 ( \Om )$. We observe that
  \begin{eqnarray}
    D_0  \left[ \grad \s- \beta \beta' (\nabla \s \cdot
    \m_{}) \otimes \m \right]  \Fsp 
    \grad \tmmathbf{\varphi} & \eqs & D_0  \left[ \grad^{\T} \s-
    \beta \beta' (\m \otimes \m^{}) \nabla^{\T}
    \s \right]  \Fsp  \grad^{\T}
    \tmmathbf{\varphi} \nonumber\\
    & \eqs & D_0 \left[ \left( I - \beta \beta' (\m \otimes
    \m^{}) \right) \grad^{\T} \s \right] 
    \Fsp  \grad^{\T} \tmmathbf{\varphi}. 
  \end{eqnarray}
  Therefore, {\eqref{eq:stspindiff}} can be written as
  \begin{equation}
    - \sum_{i = 1}^3 \partial_i \left( D_0 \left( I - \beta \beta'
    (\m \otimes \m^{}) \right) \partial_i
    \s \right) + \gamma_1 \s+ \gamma_2 \s \times
    \m= - \frac{\beta}{2} \divb (\tmmathbf{j}_e \otimes
    \m) . \label{eq:sysdivform}
  \end{equation}
  For every $\tmmathbf{\xi} \in \RR^3$, it holds that
  \begin{equation}
    [I - \beta \beta' (\m \otimes \m^{})] \tmmathbf{\xi}
    \cdot \tmmathbf{\xi} \eqs | \tmmathbf{\xi} |^2 - \beta \beta'
    (\m \cdot \tmmathbf{\xi})^2 \eqs (1 - \beta \beta') |
    \tmmathbf{\xi} |^2 + \beta \beta' | \tmmathbf{\xi} \times \m |^2\, .
  \end{equation}
  Therefore, the matrix $A_{\m} \assign D_0 [I - \beta \beta'
  (\m \otimes \m^{})]$ is uniformly positive definite, i.e.,
  \begin{equation}
    A_{\m} (x) \tmmathbf{\xi} \cdot \tmmathbf{\xi} \geqslant \gamma
    (1 - \beta \beta') | \tmmathbf{\xi} |^2 \quad \mbox{for all }\tmmathbf{\xi} \in
    \RR^3
  \end{equation}
  with $\gamma\assign \essinf_\Om \Dzero>0$.
	Hence, if we set $\tmmathbf{f}_{\m} \assign - (\beta/2)
  \divb (\tmmathbf{j}_e \otimes \m)$ and denote by
  $K_{\m}$ the matrix in $\RR^{3 \times 3}$ such that
  $K_{\m} \tmmathbf{\eta} \eqs \gamma_1 \tmmathbf{\eta}+ \gamma_2
  \tmmathbf{\eta} \times \m$ for every $\tmmathbf{\eta} \in \RR^3$,
  we can rearrange {\eqref{eq:sysdivform}} into the form of a strongly
  elliptic system:
  \begin{equation}
    - \sum_{i = 1}^3 \partial_i (A_{\m} \partial_i \s) +
    K_{\m} \s=\tmmathbf{f}_{\m} .
    \label{eq:sysdivformMcLean}
  \end{equation}
  Note that $K_{\m} \tmmathbf{\eta} \cdot \tmmathbf{\eta} \geqslant
  0$ for every $\tmmathbf{\eta} \in \RR^3$. Therefore, the bilinear form on the left-hand side of \eqref{eq:sysdivformMcLean} is uniformly elliptic in the sense of the Lax--Milgram lemma. It follows that,
  for every $\m
  \in H^1 ( \Om, \Stwo^2 )$, there exists a unique
  $\s_{\m} \in H^1 ( \Om )$ that satisfies the
  elliptic system {\eqref{eq:sysdivformMcLean}} and, therefore,  \eqref{eq:toapplyLaxMilg}. Moreover, by elliptic
  regularity {\cite[{\tmabbr{Thm.}}~4.18, p.~137]{MR1742312}}, we obtain
  $\s_{\m} \in H^{k + 2} ( \Om )$ provided that
  $A_{\m}$, $K_{\m} \in C^{k, 1} ( \overline{\Om}
  )$ and $\tmmathbf{f}_{\m} \in H^k (\Omega)$. Thus, if
  $\m, \tmmathbf{j}_e \in C^{k + 1} ( \overline{\Om} )$
  then $\tmmathbf{f}_{\m} \in C^k ( \overline{\Om} )$ and
  $\s_{\m} \in H^{k + 2} ( \Om )$. Eventually,
  by Morrey's inequality in dimension three {\cite[{\tmabbr{Thm.}}~12.55 p.~384]{MR3726909}}, we
  conclude that $\s_{\m} \in C^k ( \overline{\Om}
  )$ provided that $\m, \tmmathbf{j}_e, D_0 \in C^{k + 1} (
  \overline{\Om} )$. 
	This completes the proof.
\end{proof}

\section{Weak-Strong uniqueness of solutions (proof of Theorem
{\tmname{\ref{thm:Theorem2}}})}\label{sec:wekastrong}

\subsection{A regularity result}\label{sec:wsgensett0} The proof of the
weak-strong uniqueness of the
solutions of LLG is given in section \ref{sec:wsgensett}. Our argument allows us to obtain weak-strong uniqueness for a class of nonlinearities more general than the one introduced by $\mathcal{H}_s$. A precise definition of the type of nonlinearities covered by our result is given in the next section and is motived by the 
following result.
\begin{lemma}
  \label{lemma:regularitysm}For $\m_1 \in C^2 ( \overline{\Om},
  \Stwo^2 )$ and $\m_2 \in H^1 ( \Om, \Stwo^2 )$,
  the following estimate holds:
  \begin{equation}
    \| \mathcal{H}_s [\m_1] -\mathcal{H}_s [\m_2]
    \|^2_{H^1 ( \Om )} \; \leqslant \; c_L^2 \left(
    \| \mathcal{H}_s [\m_1] \|_{C^1 ( \overline{\Om} )}^2
    + \| \tmmathbf{j}_e \|_{C^0 ( \overline{\Om} )} \right)  \|
    \m_1 -\m_2 \|_{H^1( \Om )}^2,
    \label{eq:estimatem1m2}
  \end{equation}
  with $c_L^2$ depending only on $D_0, \beta, \beta', \gamma_1, \gamma_2$, and $\Omega$.
\end{lemma}

\begin{proof}
  With $\s_1 \assign \mathcal{H}_s [\m_1]$
  and $\s_2 \assign \mathcal{H}_s [\m_2]$, the following
  relations hold in a weak sense:
  \begin{align}
    - \divb \left[ D_0 \grad \s_1 \right] + & \beta \beta'  \divb 
    [D_0  (\nabla \s_1 \cdot \m_1) \otimes
    \m_1] \nonumber\\
    & \qquad\qquad+ \gamma_1 \s_1 + \gamma_2
    \s_1 \times \m_1  \eqs  - \frac{\beta}{2} \divb
    (\tmmathbf{j}_e \otimes \m_1)\, , \\
    - \divb \left[ D_0 \grad \s_2 \right] + &\beta \beta'  \divb 
    [D_0  (\nabla \s_2 \cdot \m_2) \otimes
    \m_2]\nonumber\\
    & \qquad\qquad + \gamma_1 \s_2 + \gamma_2
    \s_2 \times \m_2  \eqs  - \frac{\beta}{2} \divb
    (\tmmathbf{j}_e \otimes \m_2)\,  . 
  \end{align}
  We recall that the $\divb$ operator acts on columns. In a weak sense, it follows that
    \begin{align}
      &- \divb \left( D_0 \grad (\s_1 -\s_2) \right) +
      \beta \beta'  \divb [D_0  ((\nabla \s_1 \cdot \m_1)
      \otimes \m_1 - (\nabla \s_2 \cdot \m_2)
      \otimes \m_2)] \nonumber\\
      &\qquad \qquad \quad\quad + \gamma_1 (\s_1
      -\s_2) + \gamma_2 (\s_1 \times \m_1
      -\s_2 \times \m_2) \eqs - \frac{\beta}{2} \divb
      [\tmmathbf{j}_e \otimes (\m_1 -\m_2)] . \label{eq:difwf}
    \end{align}
  We note that $\s_1 \times \m_1 -\s_2 \times
  \m_2 =\s_1 \times (\m_1 -\m_2) +
  (\s_1 -\s_2) \times \m_2$, where
  the last term disappears when dot multiplied by
  $(\s_1 -\s_2)$. Hence, we have
  \begin{align}
    \int_{\Om} | [\s_1 \times \m_1 -\s_2 \times
    \m_2] \cdot (\s_1 -\s_2) | & \eqs 
    \int_{\Om} | \s_1 \cdot [(\m_1 -\m_2) \times
    (\s_1 -\s_2)] | \notag \\
    & \leqslant  \| \s_1 \|_{C(\bar \Omega)}  \| \m_1
    -\m_2 \|_{\Om}  \| \s_1 -\s_2 \|_{\Om}
    \notag \\
    & \leqslant  \frac{\| \s_1 \|_{C(\bar \Omega)}}{2 \delta^2} \|
    \m_1 -\m_2 \|_{\Om}^2 + \frac{\delta^2}{2} \|
    \s_1 -\s_2 \|_{\Om}^2 .  \label{eq:1est}
  \end{align}
  Also, we have
  \begin{align}
    & [(\nabla \s_1 \cdot \m_1) \otimes \m_1 -
    (\nabla \s_2 \cdot \m_2) \otimes \m_2]^{\T} \eqs
    (\m_1 \otimes \m_1) \nabla^{\T} \s_1 -
    (\m_2 \otimes \m_2) \nabla^{\T} \s_2 \notag \\
    & \; \eqs ((\m_1 -\m_2) \otimes \m_1) \nabla^{\T}
    \s_1  + (\m_2 \otimes \m_1) \nabla^{\T}
    \s_1  - (\m_2 \otimes
    \m_2) \nabla^{\T} \s_2 \notag \\
    &\; \eqs ((\m_1 -\m_2) \otimes \m_1)
    \nabla^{\T} \s_1 
    + (\m_2 \otimes (\m_1
    -\m_2)) \nabla^{\T} \s_1 + (\m_2 \otimes
    \m_2) \nabla^{\T} (\s_1 -\s_2) . 
    \label{eq:temptobound2}
  \end{align}
  Multiplying {\eqref{eq:temptobound2}}
  by $\tmmathbf{\varphi} \assign \s_1 -\s_2$ and applying the
  Young inequality shows for any $\delta > 0$ that
  \begin{align}
    & \left. | (\nabla \s_1 \cdot \m_1) \otimes
    \m_1 - (\nabla \s_2 \cdot \m_2) \otimes
    \m_2] \Fsp \grad (\s_1 -\s_2) \right| \nonumber\\
    & \qquad\qquad\leqslant 2 \left\| \grad \s_1 \right\|_{C(\bar \Omega)} |
    \m_1 -\m_2 |  \left| \grad (\s_1
    -\s_2) \right| + \left| \grad (\s_1 -\s_2)
    \right|^2 \nonumber \\
    & \qquad\qquad\leqslant\frac{\left\| \grad \s_1
    \right\|_{C(\bar \Omega)}^2}{\delta^2} | \m_1 -\m_2 |^2
    + (1 + \delta^2) \left| \grad (\s_1
    -\s_2) \right|^2 .  \label{eq:2est}
  \end{align}
  It follows that
    \begin{align}
      & \left| \grad (\s_1 -\s_2) \right|^2 - \beta \beta' 
      [(\nabla \s_1 \cdot \m_1) \otimes \m_1 -
      (\nabla \s_2 \cdot \m_2) \otimes \m_2]
      \Fsp \grad (\s_1 -\s_2)\nonumber\\
      & \qquad \qquad \geqslant \left| \grad (\s_1
      -\s_2) \right|^2 - \frac{\beta \beta'}{\delta^2} \left\| \grad
      \s_1 \right\|_{C(\bar \Omega)}^2 | \m_1 -\m_2 |^2 -
      \beta \beta' (1 + \delta^2) \left| \grad (\s_1
      -\s_2) \right|^2 \nonumber \\
      & \qquad \qquad \eqs (1 - \beta \beta' (1 + \delta^2))
      \left| \grad (\s_1 -\s_2) \right|^2 - \frac{\beta
      \beta'}{\delta^2} \left\| \grad \s_1 \right\|_{C(\bar \Omega)}^2 |
      \m_1 -\m_2 |^2\, .\nonumber
    \end{align}
  The first term on the right-hand side is positive for $\delta^2$
  sufficiently small. Also, the following estimate holds:
  \begin{align}
    & \left| \left\langle \divb [\tmmathbf{j}_e \otimes (\m_1
    -\m_2)], \s_1 -\s_2 \right\rangle_\Omega \right| \nonumber\\
    & \qquad\leqslant \; \int_{\Om} \left| \tmmathbf{j}_e \otimes (\m_1
    -\m_2) \Fsp \grad (\s_1 -\s_2) \right|
    + \int_{\partial \Om} | (\tmmathbf{j}_e \cdot
    \tmmathbf{n}) (\m_1 -\m_2) \cdot (\s_1
    -\s_2) | \nonumber \\
    & \qquad\leqslant \; \| \tmmathbf{j}_e \|_{C(\bar \Omega)} \| \m_1
    -\m_2 \|_{\Om} \left\| \grad (\s_1 -\s_2)
    \right\|_{\Om}  + \| \tmmathbf{j}_e \|_{C(\bar \Omega)} \| \m_1
    -\m_2 \|_{\partial \Om} \| \s_1 -\s_2
    \|_{\partial \Om} \nonumber \\
    & \qquad\leqslant \; \frac{\| \tmmathbf{j}_e \|_{C(\bar \Omega)}^2}{\delta^2} \|
    \m_1 -\m_2 \|_{H^1 ( \Om )}^2 +
    c^2_{\partial \Om} \delta^2 \| \s_1 -\s_2
    \|_{H^1 ( \Om )}^2 \nonumber
  \end{align}
  for some positive constant $c^2_{\partial \Om}$ arising from the continuity
  of the trace operator.
  
  Next, we observe that equation {\eqref{eq:difwf}} gives
\begin{align}
       &\int_{\Om} D_0 \left( \left| \grad (\s_1 -\s_2)
       \right|^2 - \beta \beta'  [(\nabla \s_1 \cdot \m_1)
       \otimes \m_1 - (\nabla \s_2 \cdot \m_2)
       \otimes \m_2] \Fsp \grad (\s_1
       -\s_2) \right) \nonumber\\
       &\quad \quad\quad \quad  + \gamma_1 \| \s_1
       -\s_2 \|_{\Om}^2+ \gamma_2 \int_{\Om} \s_1 \cdot [(\m_1
       -\m_2) \times (\s_1 -\s_2)] \nonumber\\
       & \qquad\qquad\qquad\qquad\qquad \qquad\qquad \qquad  \eqs -
       \frac{\beta}{2} \left\langle \divb [\tmmathbf{j}_e \otimes
       (\m_1 -\m_2)], \s_1 -\s_2
       \right\rangle .\nonumber
     \end{align} 
  Taking into account estimates {\eqref{eq:1est}} and
  \eqref{eq:2est} and recalling that $\inf_{\Om} D_0 \geqslant \gamma$,
 \begin{align}
       & \gamma \left( 1 - \beta \beta' (1 + \delta^2) \right)
       \left\| \grad (\s_1 -\s_2) \right\|_{\Om}^2 - \| D_0
       \|_{C(\bar\Omega)} \frac{\beta \beta'}{\delta^2} \left\| \grad \s_1
       \right\|_{C(\bar\Omega)}^2 \| \m_1 -\m_2 \|_{\Om}^2 +
       \gamma_1 \| \s_1 -\s_2 \|_{\Om}^2 \nonumber \\
       & \quad  \quad\quad  \quad\leqslant \frac{\beta}{2}
       \left[ \frac{\| \tmmathbf{j}_e \|_{C(\bar\Omega)}^2}{\delta^2} \| \m_1
       -\m_2 \|_{H^1 ( \Om )}^2 + c^2_{\partial \Om}
       \delta^2 \| \s_1 -\s_2 \|_{H^1 ( \Om
       )}^2 \right] \nonumber\\
       & \qquad  \qquad\qquad  \qquad\qquad  \qquad + | \gamma_2 | \left[ \frac{\| \s_1
       \|_{C(\bar\Omega)}^2}{2 \delta^2} \| \m_1 -\m_2 \|_{\Om}^2
       + \frac{\delta^2}{2} \| \s_1 -\s_2 \|_{\Om}^2
       \right] .
     \end{align} 
  Collecting the terms in the previous expression, we have
    \begin{align}
      & \left[ \gamma \left( 1 - \beta \beta' (1 + \delta^2)
      \right) - \delta^2  \frac{\beta}{2} c^2_{\partial \Om}  \right] \left\| \grad
      (\s_1 -\s_2) \right\|_{\Om}^2 + \left( \gamma_1 -
      \delta^2 \left( \frac{\beta}{2} c^2_{\partial \Om} + | \gamma_2
      | \right) \right) \| \s_1 -\s_2 \|_{\Om}^2 \nonumber \\
      &  \qquad  \qquad \leqslant
      \frac{1}{\delta^2} \left( \frac{| \gamma_2 | \| \s_1
      \|_{C(\bar\Omega)}^2}{2} + \| D_0 \|_{C(\bar\Omega)} \beta \beta' \left\| \grad
      \s_1 \right\|_{C(\bar\Omega)}^2 + \frac{\beta}{2} | \tmmathbf{j}_e
      |_{\infty}^2 \right) \| \m_1 -\m_2 \|_{H^1 (
      \Om)}^2 .
    \end{align} 
  Since $0 < \beta \beta' < 1$, there exists $\delta>0$ such that
  \begin{equation}
    \| \s_1 -\s_2 \|_{H^1 ( \Om )}^2 \leqslant
    c_L^2 \left( \| \s_1 \|_{C(\bar\Omega)}^2 + \left\| \grad \s_1
    \right\|_{C(\bar\Omega)}^2 + \| \tmmathbf{j}_e \|_{C(\bar\Omega)}^2 \right) \|
    \m_1 -\m_2 \|_{H^1 ( \Om )}^2
  \end{equation}
  with $c_L^2$ depending only on $D_0, \beta, \beta', \gamma_1, \gamma_2$, and $c_{\partial\Omega}$. This   concludes the proof.
\end{proof}

\subsection{Weak-strong uniqueness (energy
estimate)}\label{sec:wsgensett} Our proof of the weak-strong uniqueness of
solutions of LLG relies on Lemma \ref{lemma:forstrongweakuniq}
stated below. Our argument permits us to prove a more general form of the
weak-strong uniqueness result that we explain now. First, we deduce from the
self-adjointness of $\hd$ (cf.~\eqref{eq:hdself}) that
\begin{equation}
\frac{1}{2}\int_0^{\Ti}  \partial_t\left\langle
     \m, \hd [\m] \right\rangle_{\Om}=\int_0^{\Ti} \left\langle
     \partial_t \m, \hd [\m] \right\rangle_{\Om} \, .
\end{equation} 
Hence, integrating by parts (in time) the energy inequality \eqref{eq:eninequ0}, we infer that weak solutions of {\sllg} {\eqref{eq:CauchyequsValstationary}} satisfy, for every $\Ti > 0$, the
following form of the energy inequality:}
  \begin{align}
    \mathcal{E}{[\m]}( \Ti) & \assign \frac{c_{\tmop{ex}}}{2}
    \left\| \grad \m \left( \Ti \right) \right\|^2_{\Om} +
    \int_0^{\Ti} \alpha \| \partial_t \m \|^2_{\Om}  \notag\\
   &  \qquad  \qquad \leqslant \frac{c_{\tmop{ex}}}{2} \left\| \grad
    \m^{\ast} \right\|^2_{\Om} + \int_0^{\Ti} \left\langle
    \partial_t \m, \mu_0 \hd [\m] + j_0 \mathcal{H}_s
    [\m] \right\rangle_{\Om}  .  \label{eq:eninequalitynewforS}
  \end{align}
Second, due to Lemma \ref{lemma:regularitysm} and \eqref{eq:conthd},  we know that if $\m_1 \in C^2 ( \overline{\Om},
\Stwo^2 )$ and $\m_2 \in H^1 ( \Om, \Stwo^2 )$, then the nonlinear
operator $\tmmathbf{\pi} [\m] \assign \mu_0 \hd [\m] + j_0 \mathcal{H}_s
[\m]$ satisfies the Lipschitz-type condition
\begin{equation}
  \| \tmmathbf{\pi} [\m_1] -\tmmathbf{\pi} [\m_2]
  \|^2_{L^2 ( \Om )} \; \leqslant \; c_{\tmmathbf{\pi}}^2  \|
  \m_1 -\m_2 \|_{H^1 ( \Om )}^2 .
  \label{eq:estimatem1m2new}
\end{equation}
We stress the fact that here $c_{\tmmathbf{\pi}}$ may depend on the smooth vector field $\m_1$ (other than the physical parameters of the system) and,
therefore, condition {\eqref{eq:estimatem1m2new}} is weaker than the classical Lipschitz condition. 

Our proof of weak-strong uniqueness works in this more general setting. Therefore, for the rest of the paper, we will assume that our LLG equation has the more general form
\begin{equation}
    \partial_t \m \eqs -\m \times \Kc [\m] 
    \label{eq:CauchyequsValstationarynew}
\end{equation}
with $\Kc [\m] \assign c_{\tmop{ex}} \lapl \m+\tmmathbf{\pi}
[\m] - \alpha \partial_t \m$, and
$\tmmathbf{\pi}: H^1 ( \Om, \Stwo^2 ) \rightarrow L^2 ( \Om,
\RR^3 )$ a nonlinear operator satisfying the following
two properties:
\begin{enumerateroman}
  \item[\rm(a)] The operator $\tmmathbf{\pi}$ maps $C^{\infty} ( \overline{\Om},
  \Stwo^2 )$ into $C^{\infty} ( \overline{\Om}, \RR^3 )$.
  \smallskip
  \item[\rm(b)] For every $\m_1 \in C^{\infty} ( \overline{\Om},
  \Stwo^2 )$ and $\m_2 \in H^1 ( \Om, \Stwo^2 )$
  the Lipschitz-type condition {\eqref{eq:estimatem1m2new}} holds for some
  $c_{\tmmathbf{\pi}}$ which may depend on $\m_1$ but not on
  $\m_2$.
\end{enumerateroman}
It is important to stress that condition {\eqref{eq:estimatem1m2new}} is not
symmetric in $\m_1$ and $\m_2$ because of the special role
played by the \emph{smooth} vector field $\m_1$.

It has been already pointed out that the operator $\mu_0 \hd [\m] + j_0 \mathcal{H}_s$ satisfies (b). But it also satisfies (a): this is a consequence of Lemma \ref{lemma:regularitysm} for what concerns $ \mathcal{H}_s$, and of \cite[Lemma~8]{di2020weak} for what concerns $\hd$.
Other than for {\sllg}, the present framework covers in particular the case in which, as in {\eqref{eq:exprefffield}}, the effective
field comprises an external applied field $\tmmathbf{f}$ and the crystal
anisotropy contribution $\kappa \grad \phian (\m)$.

Generally speaking, the precise form of the LLG equation
depends on the type of interactions involved in the magnetic system. Such interactions are accounted for through the effective field $\heff$.  In its basic form, the effective field takes into account exchange interactions only, and further interactions can be described by suitable linear/nonlinear/local/non\-local terms into the effective field. In this respect, the current setting promotes a unified treatment of weak-strong uniqueness results for LLG.

We give the following definition.

\begin{definition}
  \label{def:weakllgpi} Let $\Om \subset \RR^3$ be a bounded domain, and let $\tmmathbf{\pi}: H^1 ( \Om, \Stwo^2 ) \rightarrow L^2 ( \Om,
\RR^3 )$ be a nonlinear operator satisfying the properties {\rm(a)} and {\rm(b)} above. Let $\m^{\ast} \in H^1 ( \Om, \Stwo^2 )$ be the initial
  value of the magnetization at time $t = 0$.  We say that $\m \in L^{\infty} ( \RR_+ ,
  H^1 ( \Om, \Stwo^2 ) )$ is a weak solution of the {\rm LLG}
  equation {\eqref{eq:CauchyequsValstationarynew}}, if  the following properties {\rm (i)--(iii)} hold for every $\Ti > 0$:
    \begin{enumerateroman}
    \item[\rm(i)] $\m \in H^1 ( \Om_{\Ti}, \Stwo^2
    )$ and $\m (0) =\m^{\ast}$ in the
    sense of traces,
    \item[\rm(ii)] for every  $\tmmathbf{\varphi} \in H^1 ( \Om_{\Ti})$,   
    \begin{align}
    \int_0^{\Ti} \langle \partial_t \m, \tmmathbf{\varphi}
    \rangle_{\Om} & \eqs  \int_0^{\Ti} \alpha \langle \partial_t
    \m, \tmmathbf{\varphi} \times \m \rangle_{\Om} 
		+ c_{\tmop{ex}}\left\langle \grad \m, \grad
    (\tmmathbf{\varphi} \times \m) \right\rangle_{\Om} \nonumber\\
    & \qquad\qquad -
    \int_0^{\Ti} \langle \tmmathbf{\pi} [\m], \tmmathbf{\varphi} \times
    \m \rangle_{\Om} \, ,
    \label{eq:defweaksolornew}
  \end{align}
  \item[\rm(iii)] the following energy inequality holds:
    \begin{eqnarray}
      \mathcal{E}{[\m]}( \Ti) \assign \frac{c_{\tmop{ex}}}{2}
      \left\| \grad \m \left( \Ti \right) \right\|^2_{\Om} +
      \int_0^{\Ti} \alpha \| \partial_t \m \|^2_{\Om}  \;
      \leqslant \; \frac{c_{\tmop{ex}}}{2} \left\| \grad \m^{\ast}
      \right\|^2_{\Om} + \int_0^{\Ti} \langle \partial_t \m,
      \tmmathbf{\pi} [\m] \rangle_{\Om}  \, .\label{eq:eninequalitynew}
    \end{eqnarray} 
    \end{enumerateroman}
\end{definition}
We prove Theorem
{\tmname{\ref{thm:Theorem2}}} in this more general setting. The main ingredient is contained in the next
result.

\begin{lemma}
  \label{lemma:forstrongweakuniq}Let $\m_2$ be a global weak
  solution of the {\tmem{LLG}} equation \tmem{(}in the sense of
  {\tmem{Definition~\ref{def:weakllgpi}}}\tmem{)}. Let $\m_1 \in
  C^{\infty} ( \overline{\Om}_{\Ti^{\ast}}, \Stwo^2 )$ be a strong
  solution of the {\tmem{LLG}} on $\overline{\Om_{\Ti^{\ast}}}$
	for some $\Ti^{\ast} >  0$. Set $\tmmathbf{v} \assign \m_2 -\m_1$. For
  {\tmabbr{a.e.}} $\Ti \in \RR$ such that $0 \leqslant \Ti < \Ti^{\ast}$, it holds that
  \begin{eqnarray}
    \mathcal{E}{[\tmmathbf{v}]}( \Ti ) & \leqslant & - c_{\tmop{ex}}
    \int_0^{\Ti} \left\langle \tmmathbf{v} \times \grad \mathcal{K}
    [\m_1], \grad \tmmathbf{v} \right\rangle_{\Om} 
    \nonumber\\
    &  & \qquad + \int_0^{\Ti} \left\langle \pmone - \alpha \partial_t \tmmathbf{v},
    \tmmathbf{v} \times \mathcal{K} [\m_1]\right\rangle_{\Om}  
		+ \int_0^{\Ti} \left\langle \partial_t
    \tmmathbf{v}, \pmone \right\rangle_{\Om} .  \label{eq:maineninequality}
  \end{eqnarray}
  Here, $\tmmathbf{v} \times \grad \mathcal{K} [\m_1] \assign
  (\tmmathbf{v} \times \partial_1 \mathcal{K} [\m_1], \tmmathbf{v}
  \times \partial_2 \mathcal{K} [\m_1], \tmmathbf{v} \times
  \partial_3 \mathcal{K} [\m_1])$, $\Kc [\m] \assign c_{\tmop{ex}}
  \lapl \m_{} +\tmmathbf{\pi} [\m] - \alpha \partial_t
  \m$, and $\pmone \assign \tmmathbf{\pi} [\m_1
  +\tmmathbf{v}] -\tmmathbf{\pi} [\m_1]$.
\end{lemma}

\begin{proof}
  Note that $\m_2 = \m_1
  +\tmmathbf{v}$. Since $\tmmathbf{\pi}$ is nonlinear, it is convenient
  to write
  \begin{equation}
    \tmmathbf{\pi} [\m_1 +\tmmathbf{v}] =\tmmathbf{\pi}
    [\m_1] + \pmone \, .
  \end{equation}
By the energy inequality \eqref{eq:eninequalitynew} as well as the
  boundary condition $\partial_{\tmmathbf{n}} \m_1 = 0$ on
  $\partial \Om$, it holds for every $\Ti > 0$ that
  \begin{align}
    \mathcal{E}{[\tmmathbf{v}]} ( \Ti ) & \eqs 
    \mathcal{E}{[\m_1]} ( \Ti )
    +\mathcal{E}{[\m_2]}( \Ti ) - c_{\tmop{ex}} \left\langle
    \grad \m_1 \left( \Ti \right), \grad \m_2 \left( \Ti
    \right) \right\rangle_{\Om} - 2\int_0^{\Ti} \alpha\langle \partial_t \m_1,
    \partial_t \m_2 \rangle_{\Om}  \nonumber \\
    & \overset{\eqref{eq:eninequalitynew}}{\leqslant}  -
    c_{\tmop{ex}} \left\langle \grad \m_1 \left( \Ti \right), \grad
    \m_2 \left( \Ti \right) \right\rangle_{\Om} + c_{\tmop{ex}} \left\| \grad
    \m_{\ast} \right\|^2_{\Om} + \int_0^{\Ti} \langle \partial_t
    \m_1, \tmmathbf{\pi} [\m_1] \rangle_{\Om} \nonumber\\
    &  \qquad \qquad + \langle
    \partial_t \m_2, \tmmathbf{\pi} [\m_2] \rangle_{\Om}
     - 2\int_0^{\Ti}\alpha \langle \partial_t \m_1,
    \partial_t \m_2 \rangle_{\Om} .  \label{eq:wsu1}
  \end{align}
  Let $0 < \Ti < \Ti^{\ast}$. Since
  $\m_2 \in H^1 ( \Om_{\Ti}, \Stwo^2 )$, there exists a
  family of vector-valued functions $\m_{\varepsilon} \in C^{\infty}
  ( \overline{\Om}_{\Ti}, \RR^3 )$ such that
  $\m_{\varepsilon} \rightarrow \m_2$ strongly in $H^1
  ( \Om_{\Ti}, \RR^3 )$ as $\varepsilon\to 0$. Since both $\m_1$ and
  $\m_{\varepsilon}$ are smooth, integration by parts yields that
  \begin{align}
    c_{\tmop{ex}} \left\langle \lapl \partial_t \m_1,
    \m_{\varepsilon} \right\rangle_{\Om} & \eqs  - c_{\tmop{ex}}
    \left\langle \grad \partial_t \m_1, \grad
    \m_{\varepsilon} \right\rangle_{\Om} \nonumber\\
    & \eqs  b_{1,
    \varepsilon} + c_{\tmop{ex}} \left\langle \partial_t \m_1 , \lapl
    \m_{\varepsilon}\right\rangle_\Omega\, \label{eq:byDtemp}.
  \end{align}
  where $b_{1,\varepsilon} \assign -  c_{\tmop{ex}}  \langle \partial_t \m_1, 
	\partial_{\tmmathbf{n}}\m_{\varepsilon} \rangle_{\partial \Om}$.
  Hence, the first two terms on the right-hand side of {\eqref{eq:wsu1}} read
  as
  \begin{align}
    c_{\tmop{ex}} \int_0^{\Ti} \partial_t \left\langle \lapl \m_1,
    \m_2 \right\rangle_{\Om}  & \eqs  - c_{\tmop{ex}} \left\langle
    \grad \m_1 \left( \Ti \right), \grad \m_2 \left( \Ti
    \right) \right\rangle_{\Om} + c_{\tmop{ex}} \left\| \grad \m_{\ast}
    \right\|^2_{\Om} \nonumber\\
    & \eqs  \int_0^{\Ti} \left\langle \partial_t \m_2, c_{\tmop{ex}} \lapl
    \m_1 \right\rangle_{\Om}  + c_{\tmop{ex}} \lim_{\varepsilon
    \rightarrow 0} \int_0^{\Ti} \left\langle \lapl \partial_t \m_1,
    \m_{\varepsilon} \right\rangle_{\Om}  \nonumber\\
    & \overset{\eqref{eq:byDtemp}}{\eqs}  \int_0^{\Ti} \left\langle \partial_t \m_2, c_{\tmop{ex}} \lapl
    \m_1 \right\rangle_{\Om}  + \lim_{\varepsilon
    \rightarrow 0} \int_0^{\Ti} \left(\left\langle \partial_t \m_1, c_{\tmop{ex}}
    \lapl \m_{\varepsilon} \right\rangle_{\Om} + b_{1,
    \varepsilon}\right) . \nonumber
  \end{align}
  Thus, we can rearrange {\eqref{eq:wsu1}} to obtain
  \begin{align}
    \mathcal{E}{[\tmmathbf{v}]}( \Ti ) 
    & \overset{\eqref{eq:wsu1}}{\leqslant}
      \int_0^{\Ti} \left\langle \partial_t
    \m_2, c_{\tmop{ex}} \lapl \m_1 \right\rangle_{\Om}  +
    \lim_{\varepsilon \rightarrow 0} \int_0^{\Ti} \left(\left\langle \partial_t
    \m_1, c_{\tmop{ex}} \lapl \m_{\varepsilon} \right\rangle_{\Om}
    + b_{1, \varepsilon}\right)  \nonumber\\
    & \qquad + \int_0^{\Ti} \langle \partial_t
    \m_1, \tmmathbf{\pi} [\m_1] \rangle_{\Om} + \langle
    \partial_t \m_2, \tmmathbf{\pi} [\m_2] \rangle_{\Om}
      - 2\int_0^{\Ti} \alpha\langle \partial_t
    \m_1, \alpha \partial_t \m_2 \rangle_{\Om} .   \label{eq:wsu2}
  \end{align}
  Collecting the terms in the previous expression, we arrive at
  \begin{align}
    \mathcal{E}{[\tmmathbf{v}]} ( \Ti)  & \overset{\eqref{eq:wsu2}}{\leqslant}
     \lim_{\varepsilon \rightarrow 0}
    \int_0^{\Ti} \left(\left\langle \partial_t \m_1, c_{\tmop{ex}} \lapl
    \m_{\varepsilon} +\tmmathbf{\pi} [\m_1] - \alpha
    \partial_t \m_2 \right\rangle_{\Om} + b_{1, \varepsilon}\right)
     \nonumber\\
    &   \hspace{2em} + \int_0^{\Ti} \left\langle \partial_t
    \m_2, c_{\tmop{ex}} \lapl \m_1 +\tmmathbf{\pi} [\m_2] - \alpha \partial_t
    \m_1  \right\rangle_{\Om}.  \label{eq:wsu3}
  \end{align}
 Define $\tilde{\mathcal{K}} [\m_{\varepsilon}] \assign c_{\tmop{ex}} \lapl
  \m_{\varepsilon} +\tmmathbf{\pi} [\m_2] - \alpha
  \partial_t \m_2$, where the tilde script is used to emphasize that the Laplacian
  acts on $\m_{\varepsilon}$ and not on $\m_2$. 
	Since $\partial_t \m_1$ satisfies the strong form of LLG
   {\eqref{eq:CauchyequsValstationarynew}} with $\mathcal{K}
  [\m_1] = c_{\tmop{ex}} \lapl \m_{1} +\tmmathbf{\pi}
  [\m_1] - \alpha \partial_t \m_1$ and $\tmmathbf{\pi}
  [\m_1] =\tmmathbf{\pi} [\m_2] - \pmone$, the previous
  estimate leads to
  \begin{align}
    \mathcal{E}{[\tmmathbf{v}]}( \Ti ) & \overset{\eqref{eq:wsu3}}{\leqslant}
       \lim_{\varepsilon \rightarrow 0}
    \int_0^{\Ti} \left(\left\langle \partial_t \m_1, c_{\tmop{ex}} \lapl
    \m_{\varepsilon} +\tmmathbf{\pi} [\m_2] - \alpha
    \partial_t \m_2 - \pmone \right\rangle_{\Om} + b_{1,
    \varepsilon}\right)  \nonumber\\
    &  \qquad \qquad \qquad + \int_0^{\Ti} \left\langle \partial_t
    \m_2, \mathcal{K} [\m_1] + \pmone \right\rangle_{\Om}
     \nonumber \\
    & \eqs \lim_{\varepsilon \rightarrow 0} \int_0^{\Ti}\big( \left\langle
    \partial_t \m_1, \tilde{\mathcal{K}}
    [\m_{\varepsilon}] - \pmone \right\rangle_{\Om} + b_{1,
    \varepsilon}\big)  \nonumber\\
    &   \qquad \qquad \qquad + \int_0^{\Ti} \left\langle \partial_t
    \m_2, \mathcal{K} [\m_1] + \pmone \right\rangle_{\Om}
    , \nonumber \\
    & \eqs  \lim_{\varepsilon \rightarrow 0} \int_0^{\Ti} \big(\langle
    -\m_1 \times \mathcal{K} [\m_1], \tilde{\mathcal{K}}
    [\m_{\varepsilon}] \rangle_{\Om} + \langle \partial_t
    \m_2, \mathcal{K} [\m_1] \rangle_{\Om}  +
    b_{1, \varepsilon}\big)  \nonumber\\
    &  \qquad \qquad \qquad + \int_0^{\Ti} \left\langle \partial_t
    \tmmathbf{v}, \pmone \right\rangle_{\Om} ,  \label{eq:wsu04}
  \end{align}
   On the other hand, $\m_2$ is
  a weak solution of LLG {\eqref{eq:defweaksolornew}}. From the relation
  \begin{align}
    c_{\tmop{ex}} \left\langle \grad \m_{\varepsilon}, \grad
    (\tmmathbf{\varphi} \times \m_{\varepsilon}) \right\rangle_{\Om}
    & \eqs  - c_{\tmop{ex}} \left\langle \m_{\varepsilon} \times \lapl
    \m_{\varepsilon}, \tmmathbf{\varphi} \right\rangle_{\Om} +
    b_{2,\varepsilon} [
    \tmmathbf{\varphi}] ,   \label{eq:temp06}
  \end{align}
  where $b_{2,\varepsilon} [
    \tmmathbf{\varphi}] \assign c_{\tmop{ex}} \langle \m_{\varepsilon} \times
    \partial_{\tmmathbf{n}} \m_{\varepsilon}, \tmmathbf{\varphi}
    \rangle_{\partial \Om}$,
  we infer that
  \begin{align}
    \int_0^{\Ti} \langle \partial_t \m_2, \tmmathbf{\varphi}
    \rangle_{\Om}  & \overset{\eqref{eq:defweaksolornew}}{\eqs}
      \lim_{\varepsilon \rightarrow 0}
    \int_0^{\Ti} \Bigl(\langle \alpha \partial_t \m_2, \tmmathbf{\varphi}
    \times \m_{\varepsilon} \rangle_{\Om} - \langle \tmmathbf{\pi}
    [\m_2], \tmmathbf{\varphi} \times \m_{\varepsilon}
    \rangle_{\Om} \nonumber\\
    &   \qquad\qquad\qquad\qquad\qquad\qquad \qquad \qquad +c_{\tmop{ex}} \left\langle \grad \m_{\varepsilon}, \grad
    (\tmmathbf{\varphi} \times \m_{\varepsilon}) \right\rangle_{\Om}
    \Bigr)  \notag\\
    & \eqs  \lim_{\varepsilon \rightarrow 0} \int_0^{\Ti} \Bigl(\left\langle
    \m_{\varepsilon} \times \left( \alpha \partial_t \m_2
    - c_{\tmop{ex}} \lapl \m_{\varepsilon} -\tmmathbf{\pi} [\m_2]
    \right), \tmmathbf{\varphi} \right\rangle_{\Om} +b_{2,\varepsilon}
    [\tmmathbf{\varphi}]\Bigr)  \notag \\
    & \eqs  \lim_{\varepsilon \rightarrow 0} \int_0^{\Ti} \Bigl(\langle
    -\m_{\varepsilon} \times \tilde{\mathcal{K}}
    [\m_{\varepsilon}], \tmmathbf{\varphi} \rangle_{\Om} + b_{2,\varepsilon}
    [\tmmathbf{\varphi}] \Bigr). 
    \label{eq:temp07}
  \end{align}
  To further estimate the energy $\mathcal{E}{[\tmmathbf{v}]}(\Ti)$, we use 
	$\tmmathbf{\varphi} = \mathcal{K} [\m_1]$
  in {\eqref{eq:temp07}} and plug the result in {\eqref{eq:wsu04}}. To shorten 
	the notation, we collected the two boundary terms into the term
  \begin{align}
    b_{\varepsilon} & \assign   b_{1, \varepsilon} + b_{2,\varepsilon}[ \mathcal{K} [\m_1]] \nonumber\\
    & \eqs  - c_{\tmop{ex}}  \langle \partial_t \m_1,
    \partial_{\tmmathbf{n}} \m_{\varepsilon} \rangle_{\partial \Om}
    + c_{\tmop{ex}} \langle \m_{\varepsilon} \times \partial_{\tmmathbf{n}}
    \m_{\varepsilon}, \mathcal{K} [\m_1] \rangle_{\partial \Om} . 
    \label{eq:exprbeps}
  \end{align} 
  With $\tmmathbf{v}_{\varepsilon} \assign
  \m_{\varepsilon} -\m_1$ (and observing that
  $\tmmathbf{v}_{\varepsilon} \rightarrow \tmmathbf{v}$ strongly in $H^1
  ( \Om_T, \RR^3 )$), we find that
  \begin{align}
    \mathcal{E}{[\tmmathbf{v}]} ( \Ti ) & \overset{\eqref{eq:wsu04}}{\leqslant} \lim_{\varepsilon \rightarrow 0}
    \int_0^{\Ti} \Bigl( \langle -\m_1 \times \mathcal{K} [\m_1],
    \tilde{\mathcal{K}} [\m_{\varepsilon}] \rangle_{\Om} + \langle
    -\m_{\varepsilon} \times \tilde{\mathcal{K}}
    [\m_{\varepsilon}], \mathcal{K} [\m_1] \rangle_{\Om}
     \nonumber\\
    &   \qquad \qquad \qquad \qquad \qquad\qquad \qquad \qquad \qquad \qquad \qquad+  \left\langle \partial_t
    \tmmathbf{v}, \pmone \right\rangle_{\Om} + b_{\varepsilon} \Bigr) 
    \nonumber\\
    & \eqs  \lim_{\varepsilon \rightarrow 0} \int_0^{\Ti} \Bigl( \langle
    -\m_1 \times \mathcal{K} [\m_1], \tilde{\mathcal{K}}
    [\m_{\varepsilon}] \rangle_{\Om} + \langle
    \m_{\varepsilon} \times \mathcal{K} [\m_1],
    \tilde{\mathcal{K}} [\m_{\varepsilon}] \rangle_{\Om} 
    \nonumber\\
    &   \qquad \qquad \qquad \qquad \qquad\qquad \qquad \qquad \qquad \qquad \qquad + \left\langle \partial_t
    \tmmathbf{v}, \pmone \right\rangle_{\Om} + b_{\varepsilon} \Bigr) 
    \nonumber\\
    & \eqs  \lim_{\varepsilon \rightarrow 0} \int_0^{\Ti} \left(\langle
    \tmmathbf{v}_{\varepsilon} \times \mathcal{K} [\m_1],
    \tilde{\mathcal{K}} [\m_{\varepsilon}] \rangle_{\Om} +\left\langle \partial_t \tmmathbf{v}, \pmone
    \right\rangle_{\Om} + b_{\varepsilon}\right)  \nonumber\\
    & \eqs  \lim_{\varepsilon \rightarrow 0} \int_0^{\Ti} \left(\langle
    \tmmathbf{v}_{\varepsilon} \times \mathcal{K} [\m_1],
    \tilde{\mathcal{K}} [\m_{\varepsilon}] -\mathcal{K}
    [\m_1] \rangle_{\Om}  +\left\langle \partial_t
    \tmmathbf{v}, \pmone \right\rangle_{\Om} + b_{\varepsilon}\right) . 
    \label{eq:temp08}
  \end{align}
  We now expand the quantity $\tilde{\mathcal{K}} [\m_{\varepsilon}]
  -\mathcal{K} [\m_1]$. Since $\pmone =\tmmathbf{\pi}
  [\m_2] -\tmmathbf{\pi} [\m_1]$, we have
  \begin{align}
    \tilde{\mathcal{K}} [\m_{\varepsilon}] -\mathcal{K}
    [\m_1] & \eqs  c_{\tmop{ex}} \lapl \m_{\varepsilon}
    +\tmmathbf{\pi} [\m_2] - \alpha \partial_t \m_2 - c_{\tmop{ex}}
    \lapl \m_1 -\tmmathbf{\pi} [\m_1] + \alpha \partial_t
    \m_1 \nonumber\\
    & \eqs  c_{\tmop{ex}} \lapl \tmmathbf{v}_{\varepsilon} - \alpha \partial_t
    \tmmathbf{v}+ \pmone . \nonumber
  \end{align}
  Summarizing, we have reached the following estimate:
  \begin{align}
    \mathcal{E}{[\tmmathbf{v}]} ( \Ti ) & \overset{\eqref{eq:temp08}}{\leqslant}  \lim_{\varepsilon \rightarrow
    0} \int_0^{\Ti} \left(\left\langle \tmmathbf{v}_{\varepsilon} \times \mathcal{K}
    [\m_1], c_{\tmop{ex}} \lapl \tmmathbf{v}_{\varepsilon} - \alpha
    \partial_t \tmmathbf{v}+ \pmone \right\rangle_{\Om} + b_{\varepsilon}\right)
     \nonumber\\
    &   \; \qquad\qquad\qquad+ \int_0^{\Ti} \left\langle \partial_t \tmmathbf{v}, \pmone
    \right\rangle_{\Om} \nonumber\\
    & \eqs  \lim_{\varepsilon \rightarrow 0} \int_0^{\Ti} \left(\left\langle
    \tmmathbf{v}_{\varepsilon} \times \mathcal{K} [\m_1], c_{\tmop{ex}} \lapl
    \tmmathbf{v}_{\varepsilon} \right\rangle_{\Om} + b_{\varepsilon}\right)
     \nonumber\\
    &  \;\qquad\qquad\qquad + \int_0^{\Ti} \left\langle \tmmathbf{v} \times \mathcal{K}
    [\m_1], - \alpha \partial_t \tmmathbf{v}+ \pmone
    \right\rangle_{\Om}  + \int_0^{\Ti} \left\langle \partial_t
    \tmmathbf{v}, \pmone \right\rangle_{\Om} .  \label{eq:temp09}
  \end{align}
	
  Now we have to take care of the boundary terms in
  $b_{\varepsilon}$ defined in \eqref{eq:exprbeps}. 
	An integration by parts gives
  \begin{align}
    \left\langle \tmmathbf{v}_{\varepsilon} \times \mathcal{K}
    [\m_1], c_{\tmop{ex}} \lapl \tmmathbf{v}_{\varepsilon}
    \right\rangle_{\Om} & \eqs & - c_{\tmop{ex}} \left\langle \grad
    (\tmmathbf{v}_{\varepsilon} \times \mathcal{K} [\m_1]
    \nobracket, \grad \tmmathbf{v}_{\varepsilon} \right\rangle_{\Om} + c_{\tmop{ex}}
    \int_{\partial \Om} \partial_{\tmmathbf{n}} \tmmathbf{v}_{\varepsilon}
    \cdot (\tmmathbf{v}_{\varepsilon} \times \mathcal{K} [\m_1]) . 
    \label{eq:betemp1}
  \end{align}
  Next, we expand the boundary term in the previous expression. Recalling that
  $\partial_{\tmmathbf{n}} \m_1 = 0$, we find that
  \begin{align}
    \int_{\partial \Om} \partial_{\tmmathbf{n}} \tmmathbf{v}_{\varepsilon}
    \cdot (\tmmathbf{v}_{\varepsilon} \times \mathcal{K} [\m_1]) &
    \eqs  \int_{\partial \Om} \partial_{\tmmathbf{n}}
    \m_{\varepsilon} \cdot ((\m_{\varepsilon}
    -\m_1) \times \mathcal{K} [\m_1]) \nonumber\\
    & \eqs  \int_{\partial \Om} \mathcal{K} [\m_1] \cdot
    (\partial_{\tmmathbf{n}} \m_{\varepsilon} \times
    \m_{\varepsilon}) - \int_{\partial \Om} \mathcal{K}
    [\m_1] \cdot (\partial_{\tmmathbf{n}} \m_{\varepsilon}
    \times \m_1) \nonumber\\
    & \eqs  \int_{\partial \Om} \mathcal{K} [\m_1] \cdot
    (\partial_{\tmmathbf{n}} \m_{\varepsilon} \times
    \m_{\varepsilon}) - \int_{\partial \Om} \partial_{\tmmathbf{n}}
    \m_{\varepsilon} \cdot (\m_1 \times \mathcal{K}
    [\m_1]) \nonumber\\
    & \overset{\eqref{eq:CauchyequsValstationarynew}}{\eqs} 
    - \int_{\partial \Om} (\m_{\varepsilon} \times
    \partial_{\tmmathbf{n}} \m_{\varepsilon}) \cdot \mathcal{K}
    [\m_1] + \int_{\partial \Om} \partial_{\tmmathbf{n}}
    \m_{\varepsilon} \cdot \partial_t \m_1 \nonumber\\
    & \overset{\eqref{eq:exprbeps}}{\eqs}  -
    b_{\varepsilon} / c_{\tmop{ex}} .  \label{eq:betemp2}
  \end{align}
  By {\eqref{eq:betemp1}} and {\eqref{eq:betemp2}}, we can rewrite
  {\eqref{eq:temp09}} into the form
  \begin{align}
    \mathcal{E}{[\tmmathbf{v}]}( \Ti ) \overset{\eqref{eq:temp09}}{ \leqslant} &  \lim_{\varepsilon \rightarrow
    0} \int_0^{\Ti} - c_{\tmop{ex}} \left\langle \grad (\tmmathbf{v}_{\varepsilon}
    \times \mathcal{K} [\m_1] \nobracket, \grad
    \tmmathbf{v}_{\varepsilon} \right\rangle_{\Om}  \nonumber\\
    &   \qquad + \int_0^{\Ti} \left\langle \tmmathbf{v} \times \mathcal{K}
    [\m_1], - \alpha \partial_t \tmmathbf{v}+ \pmone
    \right\rangle_{\Om}  + \int_0^{\Ti} \left\langle \partial_t
    \tmmathbf{v}, \pmone \right\rangle_{\Om} . 
  \end{align}
  Finally, we observe that, as $\varepsilon\to 0$,
  \begin{equation}
  \left\langle \grad (\tmmathbf{v}_{\varepsilon}
  \times \mathcal{K} [\m_1] \nobracket, \grad
  \tmmathbf{v}_{\varepsilon} \right\rangle_{\Om} \eqs \left\langle
  \tmmathbf{v}_{\varepsilon} \times \grad \mathcal{K} [\m_1], \grad
  \tmmathbf{v}_{\varepsilon} \right\rangle_{\Om}\to 
   \left\langle\tmmathbf{v} \times \grad \mathcal{K} [\m_1], \grad
  \tmmathbf{v} \right\rangle_{\Om} \,.
  \end{equation}
   This concludes the proof.
\end{proof}

\subsection{Proof of Theorem {\tmname{\ref{thm:Theorem2}}}} 

Thanks to the energy estimate stated in Lemma~\ref{lemma:forstrongweakuniq}, the
weak-strong uniqueness result follows from a classical argument based on the
Gronwall lemma and the Poincar{\'e} inequality. We start by the energy
inequality {\eqref{eq:maineninequality}}. In expanded form, it reads as
follows:
\begin{align}
  \frac{c_{\tmop{ex}}}{2} \left\| \grad \tmmathbf{v} \left( \Ti \right) \right\|^2_{\Om} +
  \int_0^{\Ti} \alpha \| \partial_t \tmmathbf{v} \|^2_{\Om}  &
  \leqslant  \int_0^{\Ti} \left\langle \partial_t \tmmathbf{v}, \pmone
  \right\rangle_{\Om} - c_{\tmop{ex}} \int_0^{\Ti} \left\langle \tmmathbf{v} \times
  \grad \mathcal{K} [\m_1], \grad \tmmathbf{v} \right\rangle_{\Om}
  \label{eq:finalestimatetemp1} \\
  &    \quad + \int_0^{\Ti} \left\langle
  \tmmathbf{v} \times \mathcal{K} [\m_1], \pmone \right\rangle_{\Om}
  - \int_0^{\Ti} \langle \tmmathbf{v}
  \times \mathcal{K} [\m_1], \alpha \partial_t \tmmathbf{v}
  \rangle_{\Om} .  \nonumber
\end{align}
We recall that by hypothesis
\begin{equation}
  \left\| \pmone \right\|_{\Om} \eqs \| \tmmathbf{\pi} [\m_1
  +\tmmathbf{v}] -\tmmathbf{\pi} [\m_1] \|_{\Om}
  \overset{\eqref{eq:estimatem1m2new}}{\leqslant}
  c_{\tmmathbf{\pi}} \| \tmmathbf{v} \|_{H^1 ( \Om )}\, .
\end{equation}
Setting $c_1 \assign \| \mathcal{K} [\m_1] \|_{C^1 (
\overline{\Om_T} )}$, we obtain, for every $0 < \delta < 1$, 
the following estimates:
\begin{align}
  \left| \left\langle \partial_t \tmmathbf{v}, \pmone \right\rangle_{\Om}
  \right| & \leqslant  \frac{\delta^2}{2} \| \partial_t \tmmathbf{v}
  \|_{\Om}^2 + \frac{c_{\tmmathbf{\pi}}^2}{2 \delta^2} \| \tmmathbf{v}
  \|_{\Om}^2 + \frac{c_{\tmmathbf{\pi}}^2}{2 \delta^2} \left\| \grad
  \tmmathbf{v} \right\|_{\Om}^2 \, , \\
  \left| \langle \tmmathbf{v} \times \mathcal{K} [\m_1], \alpha
  \partial_t \tmmathbf{v} \rangle_{\Om} \right| & \leqslant  \frac{\alpha
  \delta^2}{2} \| \partial_t \tmmathbf{v} \|_{\Om}^2 + \frac{c_1^2}{2 \alpha
  \delta^2} \| \tmmathbf{v} \|_{\Om}^2, \nonumber\\
  \left| \left\langle \tmmathbf{v} \times \mathcal{K} [\m_1], \pmone
  \right\rangle_{\Om} \right| & \leqslant c_1 c_{\tmmathbf{\pi}} 
	\|\tmmathbf{v} \|_{\Om}\|\tmmathbf{v} \|_{H^1(\Omega)} 
	\leqslant  c_1 c_{\tmmathbf{\pi}}\|\tmmathbf{v} \|_{\Om}^2
	+ \frac12 c_1 c_{\tmmathbf{\pi}}\|\nabla\tmmathbf{v} \|_{\Om}^2, \\
  \left| \left\langle \tmmathbf{v} \times \grad \mathcal{K} [\m_1],
  \grad \tmmathbf{v} \right\rangle_{\Om} \right| & \leqslant  
	\frac{c_1}{2} \| \tmmathbf{v} \|_{\Om}^2 
	+ \frac{c_1}{2} \left\| \grad\tmmathbf{v} \right\|_{\Om}^2 . 
\end{align}
We conclude from {\eqref{eq:finalestimatetemp1}} and the previous bounds that
\begin{align}
  0  & \leqslant  \frac{c_{\tmop{ex}}}{2} \left\| \grad \tmmathbf{v} \left( \Ti
  \right) \right\|^2_{\Om} + \alpha_{\delta} \int_0^{\Ti} \| \partial_t
  \tmmathbf{v} \|^2_{\Om}  \nonumber\\
  & \leqslant  c_{(1, \tmmathbf{\pi},
  \alpha, \delta)}^2 \int_0^{\Ti} \| \tmmathbf{v} \|_{\Om}^2  
	+ \left(\frac{c_1}{2} + \frac{c_{\tmmathbf{\pi}}^2}{2 \delta^2} 
	+ \frac12 c_1c_{\tmmathbf{\pi}}\right)
  \int_0^{\Ti} \left\| \grad \tmmathbf{v} \right\|_{\Om}^2 , 
\end{align}
with $c_{(1, \tmmathbf{\pi}, \alpha, \delta)}^2 \assign
\frac{c_{\tmmathbf{\pi}}^2}{2 \delta^2} + \frac{c_1^2}{2 \alpha \delta^2} +
\frac{c_1}{2} + c_1 c_{\tmmathbf{\pi}}$ and $\alpha^2_{\delta}
\assign \alpha \left( 1 - \frac{\delta^2}{2} \right) - \frac{\delta^2}{2}$. We
assume $\delta$ small enough so that $\alpha^2_{\delta} > 0$. 
We infer from the Poincar{\'e} inequality that
\begin{equation}
  \frac{c_{\tmop{ex}}}{2} \left\| \grad \tmmathbf{v} \left( \Ti \right) \right\|^2_{\Om} +
  \alpha^2_{\delta} \int_0^{\Ti} \| \partial_t \tmmathbf{v} \|^2_{\Om}   \leqslant  c_{(1, \tmmathbf{\pi}, \alpha,
  \delta)}^2  \Ti^2 \int_0^{\Ti} \| \partial_t \tmmathbf{v} \|^2_{\Om}  + \left(\frac{c_1}{2} + \frac{c_{\tmmathbf{\pi}}^2}{2 \delta^2} 
	+ \frac12 c_1c_{\tmmathbf{\pi}}\right) \int_0^{\Ti} \left\| \grad \tmmathbf{v} \right\|_{\Om}^2 .
  \label{eq:finestimate}
\end{equation}
Thus, for any $\Ti$ sufficiently small, say for any $\Ti < \Ti_{\ast} <
\Ti^{\ast}$ with $\Ti_{\ast}$ depending only on $c_{(1, \tmmathbf{\pi},
\alpha, \delta)}^2$, we have $\alpha_{\delta}^2 > c_{(1, \tmmathbf{\pi},
\alpha, \delta)}^2  \Ti^2$. Hence,
\begin{equation}
  \frac{c_{\tmop{ex}}}{2} \left\| \grad \tmmathbf{v} \left( \Ti \right) \right\|^2_{\Om}
  \leqslant \left(\frac{c_1}{2} + \frac{c_{\tmmathbf{\pi}}^2}{2 \delta^2} 
	+ \frac12 c_1c_{\tmmathbf{\pi}}\right) \int_0^{\Ti} \left\| \grad \tmmathbf{v} \right\|_{\Om}^2
  .
\end{equation}
According to the Gronwall lemma, we infer that $\left\| \grad \tmmathbf{v} \left( \Ti
\right) \right\|^2_{\Om} = 0$ for any $\Ti < \Ti_{\ast}$. But then, for every
$\Ti < \Ti_{\ast}$ relation {\eqref{eq:finestimate}} reduces to
\begin{equation}
  \left( \alpha_{\delta}^2 - c_{(1, \tmmathbf{\pi}, \alpha, \delta)}^2  \Ti^2
  \right) \int_0^{\Ti} \| \partial_t \tmmathbf{v} \|^2_{\Om}  \;
  \leqslant \; 0 .
\end{equation}
This implies that $\| \partial_t \tmmathbf{v} \|^2_{\Om} = 0$ in $\left( 0,
\Ti_{\ast} \right)$ and hence $\m_1 \equiv \m_2$ in
$\left( 0, \Ti_{\ast} \right)$. Since $\Ti_{\ast}$ depends essentially only on
$\Ti^{\ast}$ (more precisely, on the $C^3$ norm of the smooth function
$\m_1$, and the physical parameters of the problem), we can repeat the
argument {\tmem{finitely many times}} to conclude that $\m_1 \equiv
\m_2$ in $\left( 0, \Ti^{\ast} \right)$.

\section*{Acknowledgments}

G.~Di~F., A.~J., and D.~P. acknowledge support from the Austrian Science Fund
(FWF) through the special research program {\tmem{Taming complexity in partial
differential systems}} (Grant SFB F65). 

A.~J.\ was partially supported by the FWF, grants W1245, P30000, and P33010.

V.~S. acknowledges support by Leverhulme grant RPG-2018-438. This work was initiated when V.~S. enjoyed the hospitality of the Vienna Center for PDEs at  TU Wien.

G.~Di~F. and V.~S. would also like to thank the Max Planck Institute for Mathematics in the Sciences in Leipzig for support and hospitality.

All authors acknowledge support from ESI, the Erwin Schrödinger International
Institute for Mathematics and Physics in Wien, given in occasion of the Workshop
on New Trends in the Variational Modeling and Simulation of Liquid Crystals held at ESI, in
Wien, on December 2--6, 2019.
\bibliographystyle{siam}
\bibliography{master}

\end{document}